\documentclass[11pt, reqno]{amsart}
\usepackage[utf8]{inputenc}
\usepackage{amsmath, amsthm, amssymb, mathtools, color,verbatim,bbm}
\usepackage{tikz}
\usepackage[hidelinks]{hyperref}
\usepackage{float}
\usepackage{graphicx}   
\usepackage{subcaption}

\newcommand{\CC}{\mathbb{C}}
\newcommand{\PP}{\mathbb{P}}
\newcommand{\TT}{\mathbb{T}}
\newcommand{\Sym}{\mathbb{S}}
\newcommand{\MLD}{\mathrm{MLD}_{F}(X)}
\newcommand{\incidence}{\xi_{F}}
\newcommand{\vectorspace}{\mathcal{L}}
\newcommand{\secondvectorspace}{\mathcal{W}}
\newcommand{\bu}{u} 
\newcommand{\Co}{C_X^\circ}

\newcommand{\gencorr}{\mathfrak{N}}
\newcommand{\Con}{W(X)}

\numberwithin{equation}{section}
\theoremstyle{plain}
\newtheorem*{theorem*}{Theorem}

\newtheorem{theorem}{Theorem}
\numberwithin{theorem}{section}
\newtheorem{proposition}[theorem]{Proposition}
\newtheorem{lemma}[theorem]{Lemma}
\newtheorem{corollary}[theorem]{Corollary}

\theoremstyle{definition}
\newtheorem{definition}[theorem]{Definition}
\newtheorem{remark}[theorem]{Remark}
\newtheorem{example}[theorem]{Example}

\theoremstyle{definition}
\newtheorem*{defn*}{Definition}
\theoremstyle{plain}
\newtheorem*{thm*}{Theorem}
\theoremstyle{plain}
\newtheorem*{prop*}{Proposition}
\theoremstyle{plain}
\newtheorem*{conj*}{Conjecture}
\theoremstyle{plain}
\newtheorem*{ex*}{Example}

\definecolor{forest}{RGB}{11,128,35}

\usepackage{listings}
\definecolor{codedarkgreen}{RGB}{51, 133, 4}
\definecolor{codemaroon}{RGB}{133, 5, 63}
\definecolor{codeteal}{RGB}{0, 128, 96}
\lstdefinelanguage{Macaulay2}{
basicstyle=\small\ttfamily,
alsoletter=",
classoffset=1,
keywords={matrix,minors,gb,transpose,det,ideal,apply,subsets,ker,gens,fold,flatten,entries, numerator, denominator, diff, vars, frac,sub, map,random, sum,saturate, coefficient, multidegree, ring,min,max},
keywordstyle={\color{blue}},
classoffset=2,
morekeywords={from, to, list},
keywordstyle={\color{codemaroon}},
classoffset=3,
morekeywords={QQ, Degrees, ZZ},
keywordstyle={\color{codedarkgreen}},
classoffset=4,
morekeywords={MonomialOrder},
keywordstyle={\color{codeteal}},
xleftmargin=0.7cm,
xrightmargin=1em,
columns=fullflexible,
keepspaces=true,
stepnumber=1,
numbers=none,
captionpos=b,
showspaces=false,
frame=none
}

\title{The $F$-adjoined Gauss Map and Gaussian Likelihood Geometry}

\author{Lukas Gustafsson}

\begin{document}

\maketitle

\begin{abstract}
    We introduce the $F$-adjoined Gauss map. We use it to express the Gaussian maximum likelihood degree as a product of two invariants. As an application of our product formula, we classify all projective curves of Gaussian maximum likelihood degree 1. We also provide a formula for the \emph{generic} Gaussian maximum likelihood degree of a projective variety $X$ in terms of its polar classes.  The renowned polar class formula for generic Euclidean distance degree is a special case of our formula.
\end{abstract}

\section{Introduction}

In this paper we study algebraic statistical models of centered multivariate Gaussian distributions. These are subsets of the cone of $m\times m$ symmetric positive definite matrices given by polynomial equalities and inequalities. Each matrix $K$ in the model corresponds to the probability density function 
$$
f_K(x) = \sqrt{\frac{\det(K)}{(2\pi)^m}}e^{-\frac{x^\intercal K x}{2}}.
$$
The Gaussian maximum likelihood degree of a statistical model $\mathcal{M}$ counts the number of complex critical points of the log-likelihood function 
\begin{equation}
    \label{eqn:lS}
\ell_S(K) = \frac{n}{2} (\log \det (K) - \mathrm{trace} (SK) - m\log(2\pi))
\end{equation} 
when constrained to the complex Zariski closure $\overline{\mathcal{M}}$. Here $S=\frac{1}{n} \sum_{i=1}^n Y_i Y_i^\top$ is the \emph{sample covariance matrix} obtained from $n$ independent samples $Y_i$ of the underlying distribution and we assume throughout this article that the vanishing ideal of $\mathcal{M}$ is homogeneous. The concept of Gaussian maximum likelihood degree was introduced in \cite{SU10} for Gaussian statistical models. In \cite{ML1} a coordinate-free formulation of the problem is introduced to emphasize the underlying geometry which we now describe briefly. Given an arbitrary complex vector space $\vectorspace$, a homogeneous polynomial $F$, a general linear polynomial $\bu \in \vectorspace^\ast$ and a variety $X \subset \PP \vectorspace$ one can ask for the number of critical points of 
$$
\ell_{F, \bu}(x) = \log F(x) - \bu(x)
$$
when restricted to the smooth locus of the affine cone $C_X$ over $X$ away from the vanishing locus of $F$. The number of complex critical points is the \emph{maximum likelihood degree} of $X$ with respect to $F$ which we refer to as $\MLD$, in accordance with \cite{DRGS22}. We recover the classical Gaussian maximum likelihood degree by letting $\vectorspace$ be the set of $m \times m$ symmetric matrices $\mathbb{S}^m$, $F = \det$ and $\bu = \mathrm{trace}(S\bullet)$. In this paper we further extend the definition of maximum likelihood degree by allowing for a homogeneous rational function $F = f/g$, where $\deg(f) \neq \deg(g)$. We introduce more notation and definitions in Section \ref{sec: Notation}.

In Section \ref{sec: F-gauss}, Definition \ref{def: F-gaussmap} introduces the $F$-adjoined Gauss map of a projective variety $X$,  
\begin{displaymath}
    \gamma_{X, F} \colon X \dashrightarrow \mathrm{Gr}(\dim(X)-1, \PP\vectorspace).
\end{displaymath} 
The main goal of Section \ref{sec: F-gauss} is to use $\gamma_{X, F}$ and the order of a subvariety of the Grassmannian, introduced in Definition \ref{def: order}, to prove the following product formula for $\MLD$.

\begin{theorem}
    \label{thm: F-adjoineddegree}
Let $X$ be a projective variety and $F$ a degree $d \neq 0$ homogeneous rational function, then
    $$
\mathrm{MLD}_F(X) = \deg(\gamma^\vee_{X,F} ) \cdot \mathrm{order}(\overline{\mathrm{Im}(\gamma^\vee_{X,F} }))
    $$
    if $\gamma^\vee_{X,F}$ is generically finite, otherwise it is zero.  
\end{theorem}
It is an important problem to classify the projective varieties $X$ and homogeneous polynomials $F$ admitting $\MLD=1$ because of their connection to the \emph{maximum likelihood estimator} (MLE) of the underlying statistical model $\mathcal{M}$. An MLE of a statistical model $\mathcal{M}$ and data $u$ is a point $x \in \mathrm{argmax}_{x \in \mathcal{M}} \ell_{F, \bu}$. This concept has been extensively studied in the literature, cf. \cite{AKRS21}. Of particular interest are models whose MLE is unique and a rational function in the data, $\mathrm{MLE} \colon \vectorspace^\ast \dashrightarrow \overline{\mathcal{M}}$. If $\overline{\mathcal{M}}$ has homogeneous vanishing ideal, corresponding to a projective variety $X$, the $\mathrm{MLE}$ is a rational function in the data if and only if $\MLD = 1$ \cite[Theorem~3.1]{ML1}. The varieties $X$ admitting $\MLD =1$ are in bijection with the solutions to the \emph{homaloidal PDE} \cite[Theorem~3.5]{ML1}
\begin{equation*}\label{pde}
\Phi = F\circ (-\nabla \log \Phi),
\quad
\Phi:\vectorspace^\ast \dashrightarrow \mathbb C \text{ rational and homogeneous.}
\end{equation*}
In Section \ref{sec: classifycurves} we classify the projective curves and homogeneous polynomials $F$ such that $\MLD = 1$ using Theorem \ref{thm: F-adjoineddegree}. These curves are also referred to as curves with \emph{rational MLE}.

\begin{theorem}
\label{thm: finalcurve}
Let $X\subset \PP\vectorspace$ be a curve of degree $>1$ with vanishing ideal $\mathcal{I}(X)$ and let $X^\vee$ denote its dual variety. For any homogeneous polynomial $F$ of positive degree we have
\begin{align*}
    \MLD=1 \iff \begin{cases}
        \text{The linear span of $X$ is two-dimensional and } \\
        \exists \alpha \in \vectorspace^\ast \text{ such that: } \\
        \quad \mathrm{mult}_{[\alpha]}(X^\vee) = \deg(X^\vee) - 1, \\
        \quad F - \alpha^{\deg(F)} \in \mathcal{I}(X).
    \end{cases}
\end{align*}
Moreover, $\MLD =1$ implies that $X^\vee$ is a hypersurface $X^\vee=V(g) \subset \PP \vectorspace^\ast$ and the solution to the homaloidal PDE with respect to $F$ is determined by $\alpha$ and $g$:
\[
\Phi_{X, \vectorspace, F} = \deg(F)^{\deg(F)}\cdot \Big( \alpha(\nabla \log g) \Big)^{\mathrm{deg}(F)}.
\]
\end{theorem}

A direct consequence of this theorem is that curves $X \subset \PP\vectorspace \cong \PP^2$ with $\MLD =1$ are have bounded degree, $\deg(X) \leq \deg(F)$, as long as the radical of $F$ is not linear. This provides a classification of all varieties $X \subset \PP\Sym^2$ with $\mathrm{MLD}_{\det}(X) = 1$ by combining known results for linear subvarieties \cite[Corollary~4.12]{ML1} and Corollary \ref{cor: allS2ML1}. We also extend the study of curves with rational MLE to surfaces in Section \ref{sec: surfaceML1study} and identify a class of surfaces in $\PP^3$ with maximum likelihood degree $1$. 

Note that the Gaussian maximum likelihood degree of $X$ for the Fermat quadric $F = \sum_i x_i^2$ is the same as the Euclidean distance degree of the affine cone over $X$ \cite[Proposition~2.7]{ML1}. In Section \ref{sec: genericMLD} this connection and the ideas from earlier sections culminate in a new formula for the \emph{generic} value of $\MLD$, extending the renowned formula for the generic \emph{Euclidean distance degree} as the sum of the polar classes, $\mathrm{EDD}(X) = \sum_i \delta_i(X)$ \cite[Theorem~5.4]{DHOST16}. The formula also shares many similarities with known results for other optimization problems, cf. \cite[Theorem~1]{SS16} and \cite[Theorem~5.4]{KKS21}. 

\begin{theorem}
\label{thm: finalpolarformula}
Let $F$ be a homogeneous rational function on $\vectorspace$ of nonzero degree. Let $X \subset \PP\vectorspace$, then
$$
\mathrm{MLD}_F(X) \leq \sum_{i=0}^{\dim{\PP \vectorspace}-1} \delta_i(X) \mu_{i}(F)
$$
where $\delta_i(X)$ is the $i$'th multidegree of the conormal variety of $X$ and $\mu_i(F)$ is the maximum likelihood degree with respect to $F$ of a general $i$-dimensional projective linear subspace. Equality holds if $X$ is sufficiently general.
\end{theorem}

The exact definition for genericity in Theorem \ref{thm: finalpolarformula} is given in Definition \ref{def: F-general} and motivated by Proposition \ref{prop: pert-F-general}. Theorem \ref{thm: finalpolarformula} is intended to be used in the context of algebraic statistics and Gaussian maximum likelihood estimation, where $\vectorspace = \PP\mathbb{S}^m$ and $F = \det$. The maximum likelihood degree of a general $i$-dimensional projective linear subspace of $\mathbb{S}^m$, $\mu_i(\det)$, has been studied in the literature and is referred to as $\phi(m,i)$ \cite{MMMSV20}.

Theorem \ref{thm: finalpolarformula} brings up a future problem of better understanding the restriction of the determinant to the linear space $\vectorspace \subset \mathbb{S}^m$ associated with a graphical model which can not be considered general. Understanding the restriction of the determinant to such a special linear space would yield a better understanding of the maximum likelihood degrees of the submodels of the graphical model. Moreover, the linear space of diagonal matrices (corresponding to the graphical model with no edges) and the maximum likelihood degrees of its linear subspaces have connections to matroid theory via the \emph{permutohedral variety} \cite{DMS21,HK12} which would be interesting to explore in more generality. 

\subsection*{Acknowledgments}
The author was supported by the VR grant [NT:2018-03688]. The author would like to thank (in no particular order) Kathlén Kohn, Sandra Di Rocco, Luca Sodomaco, Luca Schaffler and Orlando Marligliano for helpful discussions and feedback.

\section{Notation}
\label{sec: Notation}
Throughout this article, $\vectorspace$ and $\secondvectorspace$ denote finite-dimensional $\CC$-vector spaces. Let $F$ be a homogeneous rational function of nonzero degree, i.e., a quotient of homogeneous polynomials of different degrees. The letter $X$ refers to an arbitrary closed irreducible projective variety $X \subset \PP(\vectorspace)$ which is not contained in the divisor associated with $F$ in $\PP \vectorspace$. Let $U(F) \subset X$ denote the open set that is the complement of the associated divisor of $F$. The projective general linear group of $\vectorspace$ is $\mathrm{PGL}(\vectorspace)$. We denote the associated affine cone over $X \subset \PP\vectorspace$ by $C_X \subset \vectorspace$. With this notation we may define 
$$
\Co := {(C_{X})_{\mathrm{reg}} \setminus \mathrm{div}(F)}
$$
where $\mathrm{reg}$ denotes the regular/smooth locus of a variety and $\mathrm{div}(F)$ the divisor associated with a rational function $F$. 

The reader is assumed to be familiar with the notion of dual variety $X^\vee$ and bi-duality for complex projective varieties covered in \cite[Chapter~1]{GKZ94}. We will denote the \emph{conormal} variety of $X$ by
$$
\Con \coloneqq \overline{\{ (x,H) \mid \, x \in X_{\mathrm{reg}} \, \text{ and } \, H \text{ is tangent to $X$ at $x$ } \}} \subset \PP \vectorspace \times \PP \vectorspace^\ast
$$
where $\vectorspace^\ast$ denotes the dual vector space whose elements are linear forms and the rational equivalence class of $W(X)$ is 
$$
[\Con] = \sum_{i=0}^{\dim \PP \vectorspace-1} \delta_i(X)H^{\dim \PP \vectorspace-i}(H^\vee)^{i+1} \in A^\ast(\PP \vectorspace \times \PP \vectorspace^\ast).
$$ 
We denote the span of two linear spaces $A, B \subset \PP \vectorspace$ by $A+B$ and for a vector $v \in \vectorspace$ we use $[v] \in \PP \vectorspace$ to denote the projective equivalence class. We refer to the gradient of a homogeneous rational function $F$ as $\nabla F \colon \vectorspace \dashrightarrow \vectorspace^\ast$ and do not distinguish it from the map
\begin{align*}
    \nabla F \colon & \PP\vectorspace \dashrightarrow \PP \vectorspace^\ast \\
    & [v] \mapsto [\nabla_v F].
\end{align*}
For $p \in \PP \vectorspace$ the hyperplane that corresponds to $\nabla_p F$ is denoted by $\nabla_p F^\vee$. For a rational map $\varphi \colon X \dashrightarrow Y$ we denote its graph by 
$$
\Gamma_\varphi = \{(x,y) \colon \varphi(x) = y \} \subset X \times Y
$$ 
and its closure by $\overline{\Gamma}_{\varphi}$. We denote the multi degrees of the gradient map of $F$ by $\mu_i(F)$, i.e., 
$$
[\overline{\Gamma}_{\nabla F}] = \sum_{j=0}^{\dim \PP\vectorspace} \mu_j(F) H^j (H^\vee)^{\dim \PP\vectorspace-j} \in A^\ast(\PP \vectorspace \times \PP \vectorspace^\ast)
$$
We denote the Grassmannian of $k$-dimensional subspaces of $\PP \vectorspace$ by 
$\mathrm{Gr}(k, \PP \vectorspace)$ and $\TT_xX$ denotes the embedded projective tangent space of $X$ at $x$. The \emph{Gauss map} is defined as
\begin{align*}
    \gamma_X \colon  X & \dashrightarrow \mathrm{Gr}(\dim(X), \PP \vectorspace) \\
  x & \mapsto \TT_xX.
\end{align*}
We identify the space $\Sym^m$ of symmetric $m \times m$ matrices with its dual space $\Sym^{m\ast}$ via the trace pairing,
$$
S \in \mathbb{S}^m \sim \Big( K \mapsto \mathrm{trace}(SK) \Big) \in \Sym^{m\ast}.
$$
We refer to elements of $\Sym^{m\ast}$ by $S$ and they correspond to matrices with entries $\{s_{ij}\}_{1\leq i\leq j \leq m}$. Elements in $\mathbb{S}^m$ are referred to as $K$ with entries $\{\kappa_{ij}\}_{1\leq i\leq j \leq m}$. With this identification the gradient of a function $F \colon \Sym^{m} \dashrightarrow \CC$ is given by
$$
(\nabla_K F)_{ij} = \frac{1}{2-\delta_{ij}} \frac{\partial F}{\partial \kappa_{ij}}
$$
where $\delta_{ij}$ is the Kronecker delta. For example on $\Sym^2$ we may consider
\begin{align*}
    F(K) = F(\begin{smallmatrix} \kappa_{11} & \kappa_{12} \\ \kappa_{12} & \kappa_{22} \end{smallmatrix} ) = \kappa_{11}\kappa_{22}-\kappa_{12}^2 - \kappa_{11}^2
\end{align*}
and its gradient is
$$
\nabla_K F = (\begin{smallmatrix} \kappa_{22}-2\kappa_{11} &  -\kappa_{12} \\ -\kappa_{12} & \kappa_{11} \end{smallmatrix} ).
$$

\section{ A product formula for Gaussian maximum likelihood degree }
\label{sec: F-gauss}

The aim of this section is to introduce the $F$-adjoined Gauss map and give a self-contained proof of the Theorem \ref{thm: F-adjoineddegree} that is needed in Section \ref{sec: classifycurves} to classify the curves with rational MLE. The theory of this section is built from ideas presented in \cite[Section~2]{DRGS22}. 

\begin{definition}
\label{def: F-gaussmap}
    The \textit{F-adjoined Gauss map} associated with a variety $X \subset \PP \vectorspace$ and homogeneous rational function $F \colon \vectorspace \dashrightarrow \CC$ of degree $d \neq 0$ is the map
    \begin{align*}
        \gamma_{X,F}\colon  X & \dashrightarrow \mathrm{Gr}(\mathrm{dim}(X) -1, \PP \vectorspace) \\
         p & \mapsto (\TT_p X) \cap  (\nabla_p F^\vee).
    \end{align*}
    Dually we can define the map
    \begin{align*}
        \gamma^\vee_{X,F} \colon  X & \dashrightarrow \mathrm{Gr}(\mathrm{codim}_{\PP \vectorspace}(X), \PP \vectorspace^\ast) \\
         p &\mapsto (\TT_p X)^\vee + \nabla_p F.
    \end{align*}
    We also define the \textit{projective maximum likelihood correspondence} as
$$
\mathfrak{X}_{F} = \{ (x, \bu) \colon \exists L \text{ s.t. } (x,L) \in \overline{\Gamma}_{\gamma^\vee_{X,F} } \, \text{ and } \, [\bu] \in L \} \subset X \times \vectorspace^\ast.
$$
\end{definition}

\begin{figure}[H]
    \centering
    \begin{subfigure}{.45\textwidth}
        \centering
    \includegraphics[width=.9\textwidth]{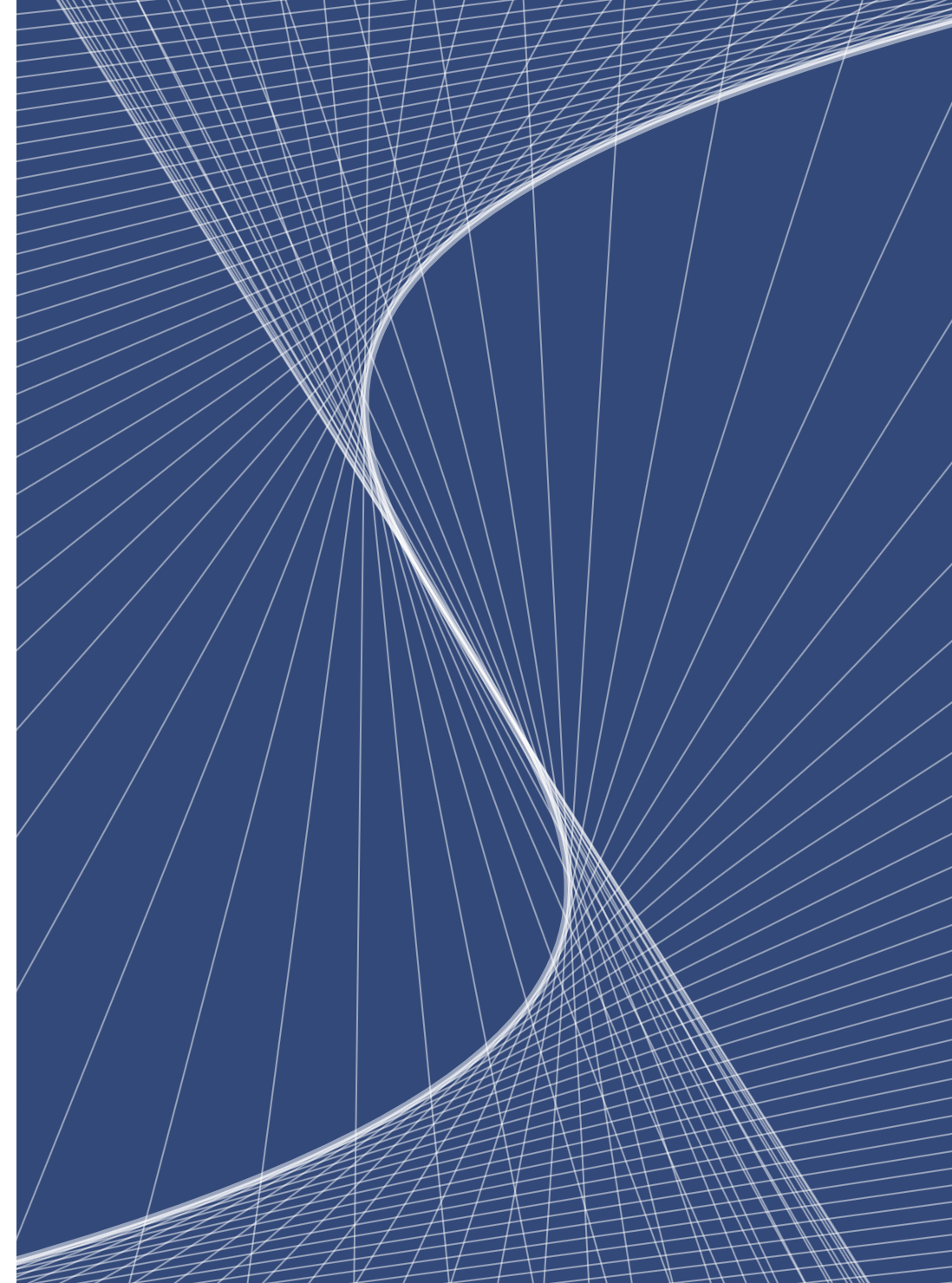}
        \caption{Illustration of $\gamma_X$ for a curve.}
        \label{fig:Gauss}
    \end{subfigure}
    \begin{subfigure}{.45\textwidth}
        \centering
    \includegraphics[width=.885\textwidth]{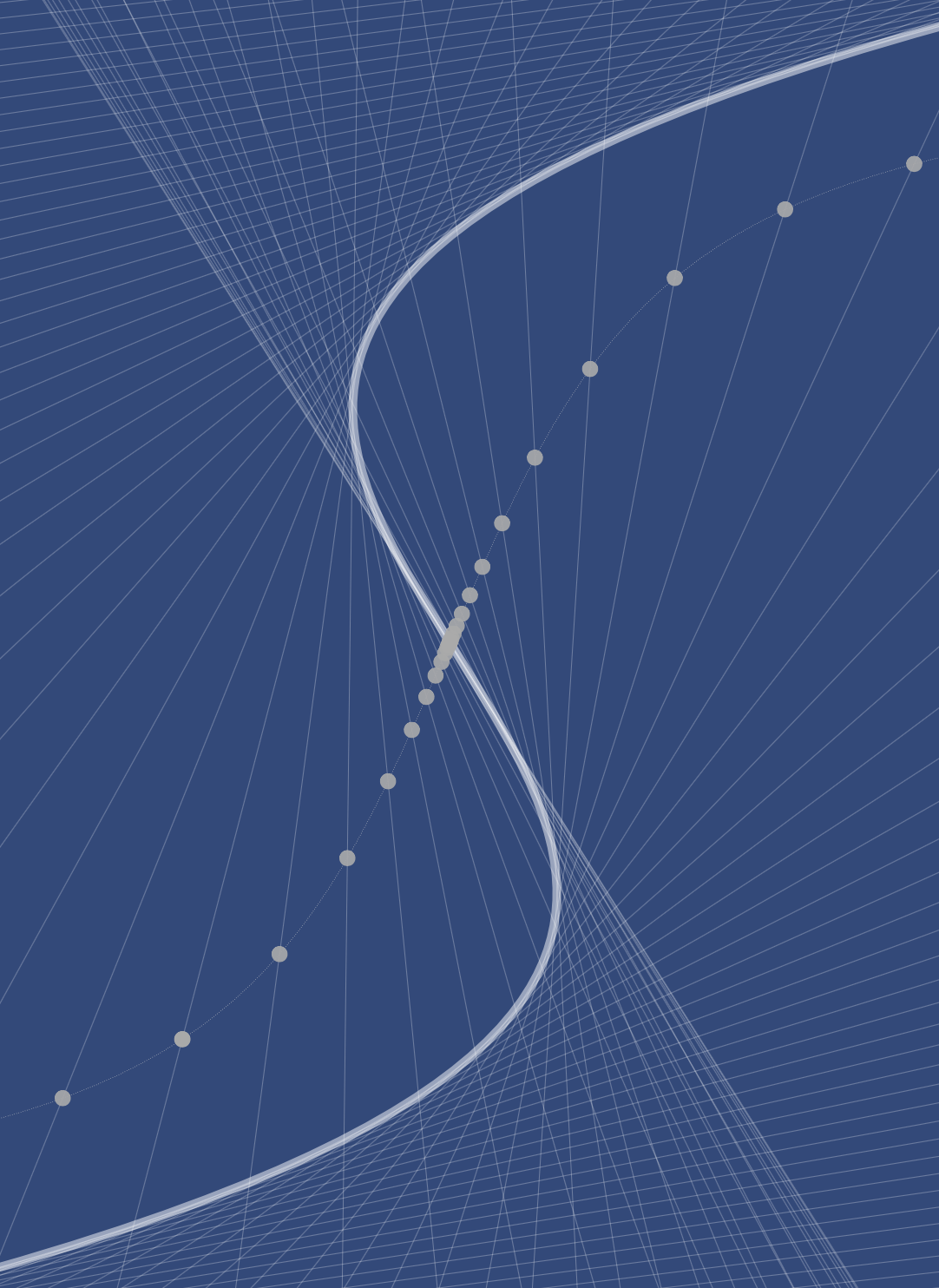}
    \caption{ Illustration of $\gamma_{X,F}$ for a curve.}
    \label{fig:F-Guass}
    \end{subfigure}
    \caption{Figure (A) illustrates a collection of tangent lines to a curve, i.e., some selected values of $\gamma_X$. Figure (B) illustrates a collection of points, each lying on a unique tangent line, that correspond to some selected values of $\gamma_{X,F}$. Images created with the Desmos Graphing Calculator, used with permission from Desmos Studio PBC \cite{desmos}.}
    \label{fig:Gaussmaps}
\end{figure}

\begin{example}
    Consider the curve $X = V(g)\subset \PP(\mathbb{C}^3)$ for $g=x_0x_1 -x_2^2-x_0^2 $ and let $F = x_0x_1-x_2^2$. Identify $\mathbb{C}^3$ with its dual vector space via the bilinear form $\langle x, y \rangle = x_0y_0+x_1y_1+x_2y_2$. Then
    $$
    \nabla_x F = (
    x_1, x_0, -2x_2) \qquad \nabla_x g = (
    x_1-2x_0, x_0, -2x_2)
$$ and 
    $$
    \TT_{x}X = \gamma_X(x) = (\nabla_x g)^\vee = \{ a\in \PP^2 \mid (x_1 - 2x_0)a_0+ x_0a_1 -2x_2a_2 = 0 \}
    $$
    This means that the image of $x$ under the $F$-adjoined Gauss map is the linear space of dimension $0 = \dim(X) -1$ obtained by intersecting $\TT_xX$ with $(\nabla_x F)^\vee$:
\begin{align*}
       &\gamma_{X,F}(x) = \TT_{x}X \cap (\nabla_x F)^\vee   \\
       =&\{ a\in \PP^2 \mid \begin{tiny}(x_1 - 2x_0)a_0+ x_0a_1 -2x_2a_2 =  x_1a_0+ x_0a_1 -2x_2a_2 = 0 \} \end{tiny} \\
          = & \{ (0:2x_2:x_0) \}. 
\end{align*}
\end{example}

\begin{definition}
Fix a homogeneous rational function $F$ of degree $d \neq 0$ on $\vectorspace$. For $\bu\in\vectorspace^\ast$ we define the \emph{log-likelihood with respect to $F$}
\[
\ell_{F,\bu}(p) \coloneqq \log(F(p)) - \bu(p).
\]
\end{definition}

\begin{remark}
    Technically the codomain of $\ell_{F, \bu}$ is $\CC/(2 \pi i \mathbb{Z})$ but for the purpose of taking derivatives this is not a problem as we can identify its differential with a map 
\begin{align*}
        d(\ell_{F, \bu})\colon & \vectorspace \dashrightarrow \vectorspace^\ast \\ 
        & p \mapsto \frac{d_p F}{F(p)} - \bu
\end{align*}
\end{remark}

\begin{definition}
Let $X\subseteq \PP(\vectorspace)$ be a projective variety. We define the \textit{affine maximum likelihood correspondence} with respect to $F$ as
\begin{align*}
    \incidence =
    \{ (x, u) \colon x \text{ is a smooth critical point of } \ell_{F, \bu}|_{\Co} \} \subset \vectorspace \times \vectorspace^\ast. 
\end{align*}
The affine maximum likelihood correspondence comes equipped with two projections denoted by $\pi_\vectorspace$ and $\pi_{\vectorspace^\ast}$ respectively. 
\end{definition}

\begin{remark}
Fix a basis for $\vectorspace$ (and its dual basis on $\vectorspace^\ast$), generators $h_i$ for $\mathcal{I}(X)$ and letting $F$ be the quotient of two relatively prime homogeneous polynomials $F = f/g$. We may use the coordinate functions $x_i$ on $\vectorspace$ and $u_i$ on $\vectorspace^\ast$ to describe the ideals $\mathcal{I}(\mathfrak{X}_F), \mathcal{I}(\overline{\incidence}) \subset \CC[x,u]$ respectively
    \begin{align*}
       \mathcal{I}(\mathfrak{X}_F) = \left[\mathcal{I}(X) + \left \langle (c+2) \text{ minors }\begin{bmatrix}
      u \\
          g\nabla_{x} f - f\nabla_{x}g  \\
          J_X(x)
      \end{bmatrix} \right \rangle \right] \colon \left[ \sqrt{ \mathcal{I}(\gamma_{X,F}) }\right]^\infty 
    \end{align*}
    \begin{align*}
      \mathcal{I}(\overline{\incidence}) = \left[\mathcal{I}(X) + \left \langle (c+1) \text{ minors }\begin{bmatrix}
          fg(\frac{\nabla_{x} F}{F(x)} - u) \\
          J_X(x)
      \end{bmatrix} \right \rangle \right] \colon  \left[ \sqrt{ \mathcal{I}(X_{\mathrm{sing}}) \cdot \mathcal{I}(fg) }\right]^\infty
    \end{align*}
    where $c := \mathrm{codim}_{\PP \vectorspace}(X)$, the rows of $J_X$ are the gradients of  $h_i$ and 
    \begin{align*}
          \mathcal{I}(\gamma_{X,F}) =& \left \langle (c+1) \text{ minors }\begin{bmatrix}
          g\nabla_{x} f - f\nabla_{x}g  \\
          J_X(x)
      \end{bmatrix} \right \rangle \\
      \mathcal{I}(X_{\mathrm{sing}}) = &  \left \langle c \text{ minors } J_X(x) \right \rangle +  \mathcal{I}(X).
    \end{align*}
\end{remark}

\begin{lemma}
\label{lem: Xf dim}
    The projection $\pi_\vectorspace\colon \incidence \to \Co$ is a vector bundle of rank $\mathrm{codim}_{\PP \vectorspace}(X)$. In particular, if $X$ is irreducible then $\incidence$ is irreducible and of dimension $\dim(\vectorspace)$, thus also $\overline{\incidence}$. 
\end{lemma}
\begin{proof}
    The fiber over a point $x$ is explicitly given by 
    $$
    \left( x, \frac{\nabla_xF}{F(x)} \right) + \{0 \} \times N^\ast_x C_X,
    $$ 
    where $N^\ast_x C_X$ is the conormal space of $C_X$ at $x$, i.e., the row space of the Jacobian matrix $J_X(x)$. The vector bundle trivializes over each open set determined by a $c \times c$ minor of $J_X(x)$ being non-zero. The curious reader can verify that the vector space structure on each fiber is given by
    $$
(x,\bu)+(x,v)=(x,\bu+v-\nabla_x F/F(x)) \qquad t\cdot(x,\bu)=(x,t\bu+(1-t)\nabla_x F/F(x)).
    $$
    Lastly, a vector bundle over an irreducible variety is irreducible and the dimension of a vector bundle is given by the sum of the dimensions of the base and the fibers. 
\end{proof}

\begin{lemma}
    If $X$ is (uni-)rational then so is $\incidence$. 
\end{lemma}
\begin{proof}
    A vector bundle of a (uni-)rational variety is (uni-)rational and we have shown the vector bundle structure of $\incidence$ in Lemma \ref{lem: Xf dim}. Assume without loss of generality that $X = V(g_1, \ldots, g_k)$ is of codimension $c$ and that the defining equations $g_i$ are ordered such that $N_x^\ast C_X = \mathrm{Span}(\{\nabla_x g_i\}_{i=1}^c)$ on an open subset $U \subset X$. This is an explicit trivialization of the conormal bundle on $U$. Suppose for some positive integer $n$ that we are given a dominant map $\Psi \colon \CC^n \dashrightarrow \Co$, then
    \begin{align*}
        \Tilde{\Psi} \colon & \CC^{n}\times \CC^{c} \dashrightarrow \vectorspace \times \vectorspace^\ast \\
        & (t, \lambda) \mapsto \left( \Psi(t), \frac{\nabla_{\Psi(t)}F}{F(\Psi(t))} + \sum_i \lambda_i \nabla_{\Psi(t)} g_i \right)
    \end{align*}
\end{proof}

If the map $\pi_{\vectorspace^\ast} \colon \incidence \to \vectorspace^\ast$ is dominant it is generically finite because the domain and codomain have the same dimension by Lemma \ref{lem: Xf dim}.

\begin{definition}
\label{def: formal-ML-deg}
    The Gaussian \textit{ML degree} of $X \subset \PP \vectorspace$ with respect to $F$, denoted by $\MLD$, is the degree of $\pi_{\vectorspace^\ast} \colon \incidence \to \vectorspace^\ast$ if it is dominant and $0$ otherwise. 
\end{definition}

We interpret $\MLD$ as the number of critical points of $\ell_{F, \bu}|_{\Co}$ for general $\bu$.

\begin{remark}
\label{rem: euler-char}
The maximum likelihood degree is a topological invariant of $\Co$. It is shown in \cite[Theorem~1.3]{DRGS22} that for two homogeneous polynomials $F,G$ of nonzero degree
     $$
     U(F) = U(G) \implies \MLD = \mathrm{MLD}_G(X). 
     $$
Moreover, the proof of this fact doesn't rely on the fact that $F,G$ are polynomials, only that they are homogeneous rational functions of nonzero degrees.
\end{remark}

\begin{lemma}
    For generic $\bu \in \vectorspace^\ast$, the critical points contributing to $\MLD$ are reduced. 
\end{lemma}
\begin{proof}
    If $\pi_{\vectorspace^\ast}$ is not dominant then $\MLD = 0$. Otherwise, by intersecting the dense open set where $\pi_{\vectorspace^\ast}$ has finite fibers (see \cite[Theorem~10.12]{Mil09}) and the dense open set where it is a smooth morphism (see \cite[Corollary~10.7]{Har77}) we have a dense open set where we have finitely many smooth points of $\xi_F$ in the fiber. 
\end{proof}

There is a closed subset of $\vectorspace^\ast$ where the fibers of $\pi_{\vectorspace^\ast}$ are not zero-dimensional and reduced. One might refer to this set as the `Gaussian ML-discriminant' in spirit of the Euclidean distance discriminant \cite{DHOST16}.

\begin{lemma}
\label{lem: uxdegf}
    The critical points of $\ell_{F, \bu}|_{\Co}$ satisfy 
    $$
    \bu(x) = \deg(F).
    $$
\end{lemma}
\begin{proof}
    Observe that 
\begin{align*}
    (x, \bu) \in \incidence \iff   \frac{\nabla_x F}{F(x)} - \bu \in N^\ast_xC_X.
\end{align*}
The vector space $N^\ast_xC_X$ is characterized as the linear forms that vanish on the embedded tangent space $T_xC_X$. Since $C_X$ is a cone we have that $x \in T_xC_X$ and thus 
\begin{align*}
    (x, \bu) \in \incidence \iff \frac{\nabla_x F}{F(x)}(x) - \bu(x) = 0 \implies \deg(F) - \bu(x) = 0 
\end{align*}
where the last implication is due to Euler's homogeneous function theorem.
\end{proof}

\begin{proposition}
\label{prop: projectiveMLcorr}
    The maximum likelihood degree of $X$ can be computed as the degree of the projection to $\vectorspace^\ast$ from the projective maximum likelihood correspondence. 
    $$
\MLD = \deg(\pi_{\vectorspace^\ast} \colon \mathfrak{X}_{F} \dashrightarrow \vectorspace^\ast)
    $$
\end{proposition}
\begin{proof}
    We show that there is a commutative diagram of projections onto $\vectorspace^\ast$
\begin{center}
\begin{tikzpicture}[every node/.style={midway}]
\matrix[column sep={6em,between origins}, row sep={2em}] at (0,0) {
        \node(A) {$\incidence$}; & \node(B) {$\mathfrak{X}_{F}$}; ; \\
        \node(C) {}; & \node(D) {$\vectorspace^\ast$};  ; \\
        };
        \draw[dashed, ->] (A) -- (B) node[above]{$\alpha$};
        \draw[->] (B) -- (D) node[right]{$\Tilde{\pi}_{\vectorspace^\ast}$};
        \draw[->] (A) -- (D) node[left]{$\pi_{\vectorspace^\ast}$} ;.
\end{tikzpicture}
\end{center}
where $\alpha$ is birational and thus by definition
$$
\MLD = \deg(\pi_{\vectorspace^\ast}) = \deg(\Tilde{\pi}_{\vectorspace^\ast} \circ \alpha) = \deg(\Tilde{\pi}_{\vectorspace^\ast} )\cdot \deg(\alpha) = \deg(\Tilde{\pi}_{\vectorspace^\ast}).
$$
To prove this we define 
\begin{align*}
    \alpha(x, \bu) = ([x], \bu) \\
    \beta([x], \bu) = (\frac{\deg(F) x}{\bu(x)}, \bu).
\end{align*}
The image of $\alpha$ is contained in $\mathfrak{X}_{F}$ because for $x \in \Co $ we have that $L = \gamma^\vee_{X,F} ([x])$ ensures that $(x,L) \in  \Gamma_{\gamma^\vee_{X,F}} $ and $[\bu] \in L$. The domain and codomain of $\alpha$ are equidimensional irreducible varieties. To complete the proof it is therefore sufficient to show that $\beta \circ \alpha = \mathrm{Id}_{\incidence}$. This is equivalent to the fact that all critical points of $\ell_{F, \bu}|_{\Co}$ satisfy $\bu(x) = \deg(F)$. This is shown in Lemma \ref{lem: uxdegf}. This concludes the proof.
\end{proof}

\begin{definition}
\label{def: order}
    Let $Y \subset \mathrm{Gr}(k, \PP^n)$ be a closed subvariety and $p \in \PP^n$ a generic point. We define the \emph{order} of $Y$ to be the cardinality
    $$
\mathrm{order}(Y) = \# (Y \cap \Sigma_k(p)) = \#\{L \in Y \mid \, p \in L \}. 
    $$
    Here $\Sigma_k(p)$ is a Schubert variety. See \cite{EH16} for more details.
\end{definition}

\begin{example}
    An example of a variety $Y\subset \mathrm{Gr}(k, \PP \vectorspace)$ of order $1$ can be constructed as the set of $k$ planes that contain a fixed $k-1$ plane $\PP \secondvectorspace \subset \PP\vectorspace$. We denote this variety by $\PP(\vectorspace / \secondvectorspace)$. Indeed, for any point $p \not \in \PP \secondvectorspace$ there is a unique $L =   p + \PP \secondvectorspace \in Y$ that is uniquely determined by the element of the projective space $\PP(\vectorspace / \secondvectorspace)$ that $p$ represents.
\end{example}

\begin{proof}[Proof of Theorem \ref{thm: F-adjoineddegree}]
We compute the maximum likelihood degree via the Proposition \ref{prop: projectiveMLcorr} as the degree of
$\pi_{\vectorspace^\ast}\colon \mathfrak{X}_{F} \dashrightarrow \vectorspace^\ast$. By definition there are $\mathrm{order}(\overline{\mathrm{Im}(\gamma^\vee_{X,F})})$ many choices of $L$ that contain a generic point $[\bu]$. Moreover, again by definition, for each such $L$ there exists either $0$ or $\deg(\gamma^\vee_{X,F})$ many choices of $x$ such that $(x,L) \in \overline{\Gamma}_{\gamma^\vee_{X,F} } $ depending on whether $\gamma^\vee_{X,F}$ is generically finite or not. This concludes the proof. 
\end{proof}

\begin{lemma}\label{lem: lindegree}
    Let $F = \alpha \in \vectorspace^\ast$ be a linear polynomial and assume that the dual variety $X^\vee$ is a hypersurface, then 
    \[
    \MLD \leq \mathrm{deg} (X^\vee).
    \]
    If $\alpha$ is generic this is an equality.
\end{lemma}
\begin{proof}
This lemma is a direct consequence of Theorem \ref{thm: finalpolarformula} but we give a direct proof. For a generic $\bu \in \vectorspace^\ast$, Proposition \ref{prop: projectiveMLcorr} and the fact $\nabla F = \alpha \not \in X^\vee$ enables us to compute $\MLD$ as the number of points in
    $$
\{ x \in X_\mathrm{reg} \colon (\TT_xX)^\vee \cap (\alpha + \bu) \neq \emptyset \}
    $$
    where $\alpha + \bu$ denotes the projective line spanned by the two points. Now we use bi-duality \cite[Theorem~1.1]{GKZ94} to see that a generic point $H \in X^\vee$ is only tangent to a unique point $x \in X_{\mathrm{reg}}$ and the linear spaces $(\TT_xX)^\vee$ rule a dense subset of $X^\vee$. If $\alpha$ is generic then the line $\alpha + \bu$ is generic and we can avoid the boundary points. So to compute the maximum likelihood degree we can equivalently count (or bound from above if $\alpha$ is not generic) the $\MLD$ through the number of points in 
    $$
X^\vee \cap (\alpha + \bu),
    $$
    which is the desired quantity by definition.
\end{proof}

\begin{remark}
    To finish this section we note that the maximum likelihood degree under appropriate assumptions is the top Segre class of the projective maximum likelihood correspondence $\mathfrak{X}_F$, seen as a \emph{cone} (e.g. a vector bundle) over $X$, in the sense of \cite[Chapter~4]{Ful98}. This is the idea that motivated Theorem \ref{thm: finalpolarformula}, although it is not part of the proof. Under some genericity assumptions $\mathfrak{X}_F$ fits in the middle of an exact sequence because $\PP(\mathfrak{X}_F)$ both contains and is `spanned by' the conormal variety of $X$ and the graph of $\nabla F|_X \colon X \dashrightarrow \PP\vectorspace^\ast$. 
\end{remark}

\section{Curves of Gaussian Maximum likelihood degree 1}
\label{sec: classifycurves}

This section is dedicated to classifying complex projective curves and homogeneous polynomials $F$ admitting $\MLD = 1$. The main idea of the classification is that for a curve with rational MLE Theorem \ref{thm: F-adjoineddegree} guarantees that the image of $\gamma_{X,F} $ is a line.

\begin{lemma}
\label{lem: ML1bounddimension}
    Let $Y \subset \PP \vectorspace$ be the variety that is ruled by $\overline{\mathrm{Im}(\gamma_{X,F} )}$ and $S_X$, $S_Y$ be the linear span of $X$ and $Y$ respectively, then
    $$
    \dim S_X \leq \dim S_Y +1.
    $$
\end{lemma}
\begin{proof}
If $X \subset S_Y$ the statement is immediate. Assume this is not the case and thus a generic point of $X$ is outside of $S_Y$. We start by studying $Z := \text{join} (X, S_Y)$, this is the closure of the union of all lines that meet $X$ and $S_Y$ in distinct points \cite[Lecture~8]{J92}. For a generic $p \in X$ 
$$
\TT_pX = p +\gamma_{X,F} (p)
$$
where $+$ denotes the linear span. By construction $\gamma_{X,F} (p) \subset S_Y $. We use the Terracini Lemma \cite[Lemma~1.11]{Ad87} to conclude that for a generic point $p \in X$
    \begin{align*}
        \dim Z &= \dim (\TT_p X +  S_Y) = \dim (\gamma_{X,F} (p) + p+  S_Y) \\
        &= \dim (p + S_Y) = \dim(S_Y) + 1.
    \end{align*}
    Moreover, taking the secant of $Z$ and applying Terracini Lemma we have that the secant variety, denoted by $Z^2$, has dimension equal to
\begin{align*}
        \dim Z^2 =& \dim (\TT_pX+ \TT_qX+ S_Y) = \dim  (\gamma_{X,F} (p)+ p+\gamma_{X,F} (q)+ q+ S_Y)  \\
        =&\dim ( p+ q+ S_Y) \leq \dim S_Y + 2 = \dim(Z) + 1.
\end{align*}
    and thus by \cite[Proposition~1.4]{Ad87} we have that $Z$ is a linear space that contains $X$, thus the linear span of $X$ must have smaller dimension than $Z$.
\end{proof}

\begin{corollary}
    \label{cor: subspacedualishyper}
    Let $X \subsetneq \PP\vectorspace$ with linear span $S_X = \PP\vectorspace$ and $\gamma_{X,F}(x) \subset \PP \secondvectorspace$ for some a linear subspace $\secondvectorspace \subsetneq \vectorspace$. Then $\PP\secondvectorspace$ is a hyperplane. 
\end{corollary}
\begin{proof}
    Let $Y$ denote the variety ruled by $\overline{\mathrm{Im}(\gamma_{X,F} )}$. By assumption we have that the linear span $S_Y$ of $Y$ satisfies $S_Y \subset \PP\secondvectorspace$. 
    By Lemma \ref{lem: ML1bounddimension} we have that 
    $$
\dim \PP\vectorspace = \dim S_X \leq \dim(S_Y)+1 \leq \dim(\PP\secondvectorspace)+1 \leq \dim \PP\vectorspace
    $$
    and thus $ \dim(\PP\secondvectorspace) = \dim \PP\vectorspace -1 $. 
\end{proof}

The next proposition reveals the relation between polynomials $F$ and hypersurfaces $X$ with rational MLE under the assumption that $\overline{\mathrm{Im}(\gamma^\vee_{X,F} )}$ is the collection of lines passing through some fixed point $[\alpha] \in \PP\vectorspace^\ast$. This will include all plane curves but also varieties of higher dimension. 

\begin{proposition}
\label{prop: subdualsML1}
Suppose that $X \subset \PP \vectorspace$ is a non-linear hypersurface, $\mathrm{MLD}_F(X) = 1$ and that $\overline{\mathrm{Im}(\gamma^\vee_{X,F} )}$ is the family of lines passing through a point $[\alpha]$, for some $\alpha \in \vectorspace^\ast$. Then $\mathrm{MLD}_\alpha(X) = 1$ and
$$
\frac{F}{\alpha^{\deg(F)}} \Big|_X =  \frac{\Phi_{X,\vectorspace, F} }{(\deg(F)\Phi_{X,\vectorspace, \alpha})^{\deg(F)}} = \text{constant}.
$$
Here $\Phi$ denotes the corresponding solution to the homaloidal PDE and $\alpha$ can be chosen such that the constant is $1$, in which case $F - \alpha^{\deg(F)} \in \mathcal{I}(X)$.  
\end{proposition}
\begin{proof}
Note that $\gamma^\vee_{X,F}(x) \ni [\alpha] \iff \gamma_{X,F}(x) \subset [\alpha]^\vee \coloneqq H$. Since $X$ is a hypersurface and $\alpha = \nabla \alpha$, because it is linear, the $\alpha$-adjoined Gauss map is given by
    \begin{align*}
    \gamma_{X,\alpha} = \TT_p X\cap H.
\end{align*}
By construction $\gamma_{X,F}  \subset \gamma_{X,\alpha} $ and they are equidimensional. Hence
$$
\gamma_{X,F}  =\gamma_{X,\alpha}.
$$
This means that $1 = \mathrm{MLD}_F(X) = \mathrm{MLD}_\alpha(X)$ by Theorem \ref{thm: F-adjoineddegree}. We now use \cite[Theorem~3.5]{ML1}. When $\MLD = 1$ there exists an inverse to the $F$-adjoined Gauss map and it determines the $\mathrm{MLE}$ up to scaling. This means that there exists a rational function $\lambda(\bu)$ such that
\begin{align*}
    \mathrm{MLE}_{X, \vectorspace, F} (\bu) &= \lambda(\bu) \cdot  \mathrm{MLE}_{X, \vectorspace, \alpha} (\bu) \iff \\
    - \nabla_{\bu} \log \Phi_{X, \vectorspace, F} &=  -\lambda(\bu) \nabla_{\bu} \log \Phi_{X, \vectorspace, \alpha}
\end{align*}
We may pair both sides of this equation with $\bu$ and use Eulers' homogeneous function theorem $\bu(\nabla_{\bu} \Phi) = \deg(\Phi)$ (see \cite[Theorem~3.1]{ML1} for details) to deduce that
$$
\deg(F) = \lambda(\bu)\cdot 1
$$
because the solutions to the homaloidal PDE with respect to a polynomial $F$ have degree equal to $-\deg(F)$. From this, we deduce that
$$
- \nabla_{\bu} \log \frac{\Phi_{X, \vectorspace, F}}{  \Phi_{X,\vectorspace, \alpha}^{\deg(F)}} = 0
$$
which implies that $\frac{\Phi_{X,\vectorspace, F}}{ (\Phi_{X,\vectorspace, \alpha})^{\deg(F)}}$ is a constant function. Lastly, considering the rational function $F/(\alpha^{\deg(F)})$ we can evaluate this at a generic point $p$ in the affine cone over $X$, written on the form 
$$
p = -\nabla_{\bu} \log \Phi_{X, \vectorspace, F} = -\deg(F) \nabla_{\bu} \log \Phi_{X, \vectorspace, \alpha}.
$$ 
By applying the homaloidal PDE to the numerator and denominator we get
$$
\frac{F(p)}{\alpha(p)^{\deg(F)}} =\frac{\Phi_{X,\vectorspace, F}}{ (\deg{F}\Phi_{X,\vectorspace, \alpha})^{\deg(F)}}
$$ 
which is constant. 
\end{proof}

The following theorem is essential for the classification of curves with rational MLE.

\begin{theorem}\cite[Theorem~6.2]{ML1} 
\label{thm: linsolution}
    Let $\vectorspace$ be a
finite-dimensional
$\CC$-vector space and $\alpha \in\vectorspace^\ast$. The solutions to the homaloidal PDE with respect to $F=\alpha$ 
$$
\Phi = \alpha(- \nabla \log \Phi)
$$
are in bijection with the varieties $X \subset \PP \vectorspace$ such that the dual variety is a hypersurface $X^\vee = V(g) \subset \PP \vectorspace^\ast$ and $\alpha$ represents a point $[\alpha] \in X^\vee$ of multiplicity $\deg(g) -1$. The solution corresponding to $X$ is given by 
$$
\Phi_{X, \vectorspace, \alpha} = \alpha( \nabla \log g). 
$$
\end{theorem}

\begin{remark}
\label{rem: multiplicity-point}
    A convenient way of determining if $\ell$ is a point of multiplicity $\deg(g) -1$ on $V(g)$ is to verify that 
    $$
\frac{d}{dt^2} g(u + t\ell) = 0
    $$
This is an important part of the proof of the above theorem. For example the cuspidal cubic $g(u_0,u_1,u_2) = 4u_0^3-54u_0u_1^2+27u_1^2u_2 $ is linear in $u_2$, meaning that $(0:0:1)$ is point of multpilicty $3-1=2$. In general we can always perform a linear change of coordinates such that $\ell = (0: \ldots:0:1)$ and then we need only check that $g$ is linear in the last variable. 
\end{remark}

\begin{proof}[Proof of Theorem \ref{thm: finalcurve}]
The implication right-to-left is an immediate consequence of Remark \ref{rem: euler-char} because $U(F) = U(\alpha)$ combined with Theorem \ref{thm: linsolution} implies $\MLD = \mathrm{MLD}_\alpha(X) = 1$. For the other direction we note that 
$$
[\bu] \in \gamma_{X,F}^\vee(x) \iff \gamma_{X,F}(x) \in [\bu]^\vee,
$$
thus the order of $\overline{\mathrm{Im}(\gamma_{X,F}^\vee)}$ is the degree of the curve $\overline{\mathrm{Im}(\gamma_{X,F})} \subset \mathrm{Gr}(0, \PP\vectorspace) = \PP\vectorspace$ which is $1$, i.e., it is a line. If $X \subset \PP\vectorspace \cong \PP^2$ then $\overline{\mathrm{Im}(\gamma_{X,F})} = [\alpha]^\vee$ for some $\alpha \in \vectorspace^\ast$ and the statement follows from applying Proposition \ref{prop: subdualsML1} and Theorem \ref{thm: linsolution}. If the ambient projective space has higher dimension we know from Lemma \ref{lem: ML1bounddimension} that $X$ has a 2-dimensional projective linear span $\PP\secondvectorspace$ such that $X \subset \PP\mathcal{W} \subsetneq \PP\vectorspace$. We may apply Proposition \ref{prop: subdualsML1} and Theorem \ref{thm: linsolution} to $X \subset \PP\secondvectorspace$ and $F|_\secondvectorspace$ because $\MLD = \mathrm{MLD}_{F|_\secondvectorspace}(X)$ by Remark \ref{rem: euler-char}. This provides us with $\Tilde{\alpha} \in \secondvectorspace^\ast$ and a solution to the homaloidal PDE on $\secondvectorspace^\ast$,
$$
\Phi_{X, \secondvectorspace, F|_\secondvectorspace} = (\deg(F)\cdot \Phi_{X, \secondvectorspace, \Tilde{\alpha}})^{\deg(F)}.
$$
We now extend the solution $\Phi_{X, \secondvectorspace, F|_\secondvectorspace}$ to a solution $\Phi_{X, \vectorspace, F}$. Consider the two different embeddings $X \subset \PP\secondvectorspace \subset \PP\vectorspace$ and their corresponding dual varieties: $X_\secondvectorspace^\vee = V(\Tilde{g}) \subset \PP \secondvectorspace^\ast$ and $X_\vectorspace^\vee = V(g) \subset \PP \vectorspace^\ast$. Let $\pi \colon \vectorspace^\ast \to \secondvectorspace^\ast$ denote the dual map of the embedding $\iota: \secondvectorspace \to \vectorspace$. Notice that we may choose $g = \Tilde{g} \circ \pi$ and $\alpha$ such that $\Tilde{\alpha} = \alpha \circ \iota = \alpha|_\secondvectorspace$. We see that $(F - \alpha^{\deg(F)})|_X = (F|_{\secondvectorspace} - \Tilde{\alpha}^{\deg(F)})|_X = 0$. By an appropriate choice of basis of $\vectorspace$ and applying Remark \ref{rem: multiplicity-point} we deduce that $[\alpha]$ is a point on $X_\vectorspace^\vee$ of multiplicity $\deg(X_\vectorspace^\vee)-1$ because $[\Tilde{\alpha}]$ is for $X_{\secondvectorspace}^\vee$. This concludes the proof. The curious reader can verify that $\Phi_{X, \vectorspace, F}$ has the desired form by \cite[Proposition~2.8]{ML1},
$$
\Phi_{X,\vectorspace, F} = \Phi_{X, \secondvectorspace, F|_\secondvectorspace} \circ \pi.
$$
By the chain rule and Theorem \ref{thm: linsolution}
\begin{align*}
    \alpha(\nabla_{\bu} \log g ) =& \alpha(\nabla_{\bu} \log (\Tilde{g} \circ \pi)) = \alpha( (\nabla_{\pi(\bu)} \log \Tilde{g}) \circ \pi) = \alpha( \iota(\nabla_{\pi(\bu)} \log \Tilde{g}) ) \\
    =& \Tilde{\alpha}(\nabla_{\pi(\bu)} \log \Tilde{g} ) = (\Phi_{X, \secondvectorspace, F|_\mathcal{W}} \circ \pi)(\bu).
\end{align*}
Thus $ \Phi_{X,\vectorspace, F}(\bu)$ is obtained by taking a power of $\alpha(\nabla_{\bu} \log g )$.
\end{proof}

\begin{corollary}
    \label{cor: allS2ML1}
    Let $X\subset \PP\Sym^2$ be a curve of degree $>1$. Then 
    $$
    \mathrm{MLD}_{\det}(X) = 1 \iff X = V(\det(K) - \mathrm{trace}(AK)^2)  
    $$
    where $A = (\begin{smallmatrix}
        a_{11} & a_{12} \\
        a_{12} & a_{22}
    \end{smallmatrix})$ is any symmetric matrix of rank $1$. 
    The corresponding solution to the homaloidal PDE is given by 
    $$
\Phi_{X, \Sym^2, \det}(S) = 4\left(\frac{a_{22}s_{11} + a_{11}s_{22}- 2a_{12}s_{12}}{(s_{11}s_{22} - s_{12}^2) + (a_{22}s_{11} + a_{11}s_{22}- 2a_{12}s_{12})^2 }\right)^2
    $$
\end{corollary}
\begin{proof}
    We first argue that a nonlinear curve with rational MLE is of the proposed form. Theorem \ref{thm: finalcurve} ensures the existence of a linear form $\alpha(K) = \mathrm{trace}(AK)$ such that $Q = V(\det(K) - \mathrm{trace}(AK)^2)$ contains $X$ as a component. If $Q$ is not irreducible the components are linear but the curve $X$ is not, therefore $Q$ must be irreducible and $X=Q$. Notice that
    $$
    X \cap V(\mathrm{trace}(AK)) = V(\det(K),\mathrm{trace}(AK) ).
    $$  
    By Theorem \ref{thm: finalcurve} $\mathrm{mult}_{[\mathrm{trace}(AK)]}X^\vee = 1$, which means that it corresponds to a tangent line of $X$, therefore it is also tangent to $V(\det)$, i.e., $\mathrm{rank}(A)=1$. All curves of this form have rational MLE since they satisfy the conditions of Theorem \ref{thm: finalcurve}. We also provide the corresponding solution to the homaloidal PDE by computing the dual quadric of $X$. The dual quadric is in the denominator of $\Phi_{X, \Sym^2, \det}$, 
    $$
Q^\vee(S) = \det(S) + \mathrm{trace}(\mathrm{adj}(A)S)^2,
    $$
    where $\mathrm{adj}(A) = (\begin{smallmatrix}
        a_{22} & -a_{12} \\
        -a_{12} & a_{11}
    \end{smallmatrix})$. To see this it is enough to show that the gradients
    \begin{align*}
        \nabla_KQ &= \mathrm{adj}(K) - 2\mathrm{trace}(AK)A \\
        \nabla_SQ^\vee &= \mathrm{adj}(S) + 2\mathrm{trace}(\mathrm{adj}(A)S)\mathrm{adj}(A)
    \end{align*}
    are mutually inverse (up to scaling). The map $A \mapsto \mathrm{adj}(A)$ is linear and satisfies 
    \begin{align*}
        \mathrm{trace}(\mathrm{adj}(A)S) =& \mathrm{trace}(A\mathrm{adj}(S)) \\
        \mathrm{adj}(\mathrm{adj}(A)) =& A \\
        \mathrm{rank}(A)=1 \implies &\mathrm{trace}(\mathrm{adj}(A)A)=0.
    \end{align*} 
    We then compute 
\begin{align*}
       (\nabla  Q^\vee \circ \nabla Q)(K)   =& \mathrm{adj}\Big(\mathrm{adj}(K) - 2\mathrm{trace}(AK)A \Big) \\
       &+ 2\mathrm{trace}\Big(\mathrm{adj}(A)[\mathrm{adj}(K) - 2\mathrm{trace}(AK)A]\Big)\mathrm{adj}(A) \\
       =& K - 2\mathrm{trace}(AK)\mathrm{adj}(A)+ 2\mathrm{trace}(\mathrm{adj}(A)\mathrm{adj}(K))\mathrm{adj}(A) \\
       =&K.
    \end{align*}
    Theorem \ref{thm: linsolution} provides us with the solution 
    \begin{align*}
        \Phi_{X, \Sym^2, \det}(S) =& \mathrm{trace}\left(A \nabla_S\log Q^\vee \right)  \\
        =&4\left(\frac{\mathrm{trace}\left(A[\mathrm{adj}(S) + 2\mathrm{trace}(\mathrm{adj}(A)S)\mathrm{adj}(A)] \right)}{\det(S) + \mathrm{trace}(\mathrm{adj}(A)S)^2}\right)^2 \\
        =& 4\left(\frac{\mathrm{trace}(\mathrm{adj}(A)S)}{\det(S) + \mathrm{trace}(\mathrm{adj}(A)S)^2}\right)^2
    \end{align*}
\end{proof}

\begin{example}
\label{ex: 2x2ML1 quadric}
Consider the curve $X= V(\kappa_{11}\kappa_{22}-\kappa_{12}^2 - \kappa_{11}^2) \subset \PP \mathbb{S}^2$. We may apply Theorem \ref{thm: finalcurve} or Corollary \ref{cor: allS2ML1} using the linear form 
$$
\alpha(K) = \mathrm{trace}\left( \left(\begin{smallmatrix} 1 & 0 \\ 0 & 0 \end{smallmatrix}\right) \left(\begin{smallmatrix}\kappa_{11} & \kappa_{12} \\ \kappa_{12} & \kappa_{22} \end{smallmatrix}\right) \right) = \kappa_{11}
$$
which is corresponds to a tangent line of $X$, more specifically it is a smooth (multiplicity $1$) point of the dual variety:
$$
A = \left(\begin{smallmatrix} 1 & 0 \\ 0 & 0 \end{smallmatrix}\right) \in X^\vee = V(s_{11}s_{22} - s_{12}^2 + s_{22}^2).
$$
Theorem \ref{thm: finalcurve} and Corollary \ref{cor: allS2ML1} provide the correponding solution to the homaloidal PDE
 $$
 \Phi_{X,\mathbb{S}^2, \det}(\begin{smallmatrix} s_{11} & s_{12} \\ s_{12} & s_{22} \end{smallmatrix} ) = 4\left( \frac{s_{22}}{s_{11}s_{22} - s_{12}^2 + s_{22}^2}\right)^2.
 $$
We now verify that this is a solution to the homaloidal PDE with respect to the determinant and that $X$ is its associated variety. Recall that the pairing $\mathrm{trace}(AB)$ lets us identify the vector space of symmetric matrices with its dual vector space as long as we scale the off-diagonal entries by $\frac{1}{2}$. We may compute $(\nabla (\Phi_{X,\mathbb{S}^2, \det}))_{ij} = \frac{1}{2- \delta_{ij}} \frac{\partial \Phi_{X,\mathbb{S}^2, \det}}{\partial s_{ij}}$ where $\delta_{ij}$ is the Kronecker delta, 

\begin{align*}
    -\nabla_S \log (\Phi_{X,\mathbb{S}^2, \det}) = &\frac{2}{s_{11}s_{22} - s_{12}^2 + s_{22}^2}\left(\begin{smallmatrix} s_{22} & -s_{12} \\ -s_{12} & s_{11} + 2s_{22} \end{smallmatrix}\right) - \frac{2}{s_{22}}\left(\begin{smallmatrix} 0 & 0 \\ 0 & 1\end{smallmatrix}\right) \\
    =&\frac{2}{s_{22}(s_{11}s_{22} - s_{12}^2 + s_{22}^2)}\left(\begin{smallmatrix} s_{22}^2 & -s_{12}s_{22} \\ -s_{12}s_{22} &   s_{12}^2 +  s_{22}^2\end{smallmatrix}\right)  = \mathrm{MLE}(S).
\end{align*}
We verify that
\begin{align*}
    \det(-\nabla \log \Phi_{X, \mathbb{S}^2, \det}) =& 4\frac{ s_{22}^2(s_{12}^2 +  s_{22}^2) -  s_{12}^2s_{22}^2 }{s_{22}^2(s_{11}s_{22} - s_{12}^2 + s_{22}^2)^2}= 4\left( \frac{s_{22}}{s_{11}s_{22} - s_{12}^2 + s_{22}^2}\right)^2.
\end{align*}
The associated variety of this solution is the image of the $\mathrm{MLE} \colon \Sym^2 \dashrightarrow \Sym^2$,
$$
\mathrm{MLE} = -\nabla \log (\Phi_{X, \mathbb{S}^2, \det}) \propto \left(\begin{smallmatrix} s_{22}^2 & -s_{12}s_{22} \\ -s_{12}s_{22} &   s_{12}^2 +  s_{22}^2\end{smallmatrix}\right).
$$
From this we can deduce that the entries of the $\mathrm{MLE}$ map satisfy the equation
$$
\kappa_{11}\kappa_{22}-\kappa_{12}^2 - \kappa_{11}^2 = 0
$$
which defines the variety $X$. 
\end{example}

\begin{example}
\label{ex: cuspML1}
   Let $\vectorspace =\CC^3$ and consider $X = V\left(  y^2z-(x+2z)^3 \right)$. Recall that its dual variety
$$
X^\vee = V( 4u^3-54uv^2+27v^2w )
$$
is a cuspidal curve. The polynomial $g=4u^3-54uv^2+27v^2w$ is linear in $w$. By Remark \ref{rem: multiplicity-point} the point of multiplicity $2$ (the cusp) is $V(u,v) \in X^\vee$. This point $V(u,v)$ corresponds to the line $z=0$ that is tangent to $X$ at the inflection point $V(x,z) = (0:1:0) \in X$. By Theorem \ref{thm: linsolution} we may choose the linear form $\alpha(x,y,z) = 8z$ as a representative of this tangent line to obtain a solution
$$
\Phi_{X, \CC^3, 8z} = 8\frac{\partial}{\partial w} \log ( 4u^3-54uv^2+27v^2w ) = 8\frac{27v^2}{4u^3-54uv^2+27v^2w }.
$$
 to the homaloidal PDE with respect to $F = 8z$, i.e., 
 $$
\Phi = -8\frac{\partial}{\partial w} \log \Phi.
 $$
Now consider 
$$
F =y^2z -x^3 - 6x^2z - 12xz^2 +504z^2
$$
and note $F - (8z)^3 \in \mathcal{I}(X)$. Theorem \ref{thm: finalcurve} then implies that $\mathrm{MLD}_F(X)=1$ and
$$
\Phi_{X, \CC^3, F} = 3^3\Phi_{X, \vectorspace, 8z}^3.
$$
We also verify this,
\begin{align*} 
    F(-\nabla \log \Phi_{X, \CC^3, F}) = & F(- \nabla \log 3^3\Phi_{X, \vectorspace, 8z}^3) = 3^3 F(- \nabla \log \Phi_{X, \vectorspace, 8z}) \\
    \{ F|_X = (8z)^3 \} =& 3^3 (-8\frac{\partial}{\partial w}\log \Phi_{X, \vectorspace, 8z}))^3 = 3^3\Phi_{X, \vectorspace, 8z}^3 = \Phi_{X, \CC^3, F}.
\end{align*}
\end{example}

\begin{example}
Using Example \ref{ex: cuspML1} we construct a curve inside the space of $3 \times 3$ matrices $\vectorspace = \mathbb{S}^3$ that has maximum likelihood degree $1$. Let $\secondvectorspace$ be the sub-vector space spanned by the matrices
$$
A =\left(\begin{smallmatrix} 0 & 0 & 1 \\ 0&1&0 \\ 1& 0 & 0\end{smallmatrix}\right), B= \left(\begin{smallmatrix} 0 & 0 & 0 \\ 0&0&1 \\ 0& 1 & 0\end{smallmatrix}\right), C= \left(\begin{smallmatrix} -1 & 0 & 6 \\ 0&-6&0 \\ 6& 0 & 48\end{smallmatrix}\right).
$$
Let $F$ be the determinant and $X$ be a cuspidal curve inside of
$\secondvectorspace$ defined as the matrices $xA+yB+zC$ such that $  y^2z-(x+2z)^3 = 0$, or equivalently 
$$
X = {\scriptstyle V(4\kappa_{22}-4\kappa_{13}+\kappa_{33},\kappa_{12},48\kappa_{11}+\kappa_{33},1728\kappa_{13}^3-432\kappa_{13}^2\kappa_{33}-36\kappa_{23}^2\kappa_{33}+36\kappa_{13}\kappa_{33}^2-\kappa_{33}^3)}.
$$ 
Notice that 
$$
\det(xA+yB+zC) = F(x,y,z) = y^2z -x^3 - 6x^2z - 12xz^2 +504z^2.
$$
We observe that this the Example \ref{ex: cuspML1} in disguise. We use the explicit bijection $\secondvectorspace \cong \CC^3$ given by the matrices $A, B, C$. Applying \cite[Proposition~2.8]{ML1} proves that $X$ has maximum likelihood degree $1$ and we obtain the following factorization 
$$
\Phi_{X, \mathbb{S}^3, \det}(S) = \Phi_{X, \CC^3, F}^3 \circ \pi_{\CC^3}
$$
where $\pi_{\CC^3}$ is defined by restricting dual vectors to $\secondvectorspace$. More explicitly:
\begin{align*}
    \pi_{\CC^3} (S) =&  \mathrm{trace}\left( S(xA+yB+zC) \right)  \\ =& x \cdot \mathrm{trace}(SA)+y\cdot \mathrm{trace}(SB)+z \cdot \mathrm{trace}(SC) \\ 
    =& (2s_{13} + s_{22})x + 2s_{23}y + (-s_{11} - 6s_{22} + 48s_{33} + 12s_{13})z.
\end{align*}
The linear forms $x,y,z$ form the basis of $\secondvectorspace^\ast$ that is dual to the basis $A, B, C$ of $\secondvectorspace$. So we now substitute $u,v,w$ by $\mathrm{tr}(SA),\mathrm{tr}(SB), \mathrm{tr}(SC)$ into the known solution $\Phi_{X, \CC^3, F}$ from Example \ref{ex: cuspML1} to obtain the solution, 

\begin{align*}
    \Phi_{X, \mathbb{S}^3, \det}  =  {\scriptstyle 2^93^3\left(\frac{27(2s_{23})^2}{4(2s_{13} + s_{22})^3-54(2s_{13} + s_{22})(2s_{23})^2+27(2s_{23})^2(-s_{11} - 6s_{22} + 48s_{33} + 12s_{13}) } \right)^3}.
\end{align*}

This solution can be verified using the code provided in \cite{ML1} at \url{https://mathrepo.mis.mpg.de/GaussianMLDeg1}.

\end{example}

\section{A family of maximum likelihood degree 1 surfaces}
\label{sec: surfaceML1study}

Theorem \ref{thm: finalpolarformula} suggests a suitable class of varieties that are not curves and that might have rational MLE. These are the dual varieties of curves, as they have a small number of nonzero polar degrees. We explore this idea for surfaces in $\PP^3$ in this section. The easiest cases of surfaces that are dual to a plane curve. These are addressed in the following Lemma.

\begin{lemma}
\label{lem: dual-planecurve}
    Suppose $X \subset \PP \vectorspace$, $Y \subset \PP \secondvectorspace$ and $\MLD = \mathrm{MLD}_G(Y) =1 $. Then consider the join $Z_{X,Y} \subset \PP(\vectorspace \oplus \secondvectorspace)$, cut out by the equations of both $X$ and $Y$. Then $\mathrm{MLD}_{F \cdot G}(Z_{XY}) = 1$ and 
    $$
    \mathrm{MLE}_{Z_{XY},\vectorspace \oplus \secondvectorspace, F\cdot G} = \mathrm{MLE}_{X,\vectorspace, F} \oplus \mathrm{MLE}_{Y,\secondvectorspace , G}
    $$
    or equivalently 
    $$
    \Phi_{Z_{XY},\vectorspace \oplus \secondvectorspace, F\cdot G}(\bu_\vectorspace, \bu_\secondvectorspace) = \Phi_{X,\vectorspace, F}(\bu_\vectorspace) \cdot \Phi_{Y,\secondvectorspace, G}(\bu_\secondvectorspace).
    $$
\end{lemma}
\begin{proof}
This is an immediate consequence of plugging the suggested solution into the homaloidal PDE because it splits as $F$ does not depend on the entry in $\secondvectorspace$ and $G$ not on the entry in $\vectorspace$. 
\end{proof}

\begin{example}
\label{ex: dual-to-planecurve}
    Consider the singular quadratic surface $Q = V(\kappa_{11}\kappa_{22}- \kappa_{12}^2 - \kappa_{11}^2, \kappa_{13}, \kappa_{23}) \subset \PP \mathbb{S}^3$ sitting inside the linear space $\PP V$ of matrices of the form 
$$
\left(\begin{smallmatrix} \kappa_{11} & \kappa_{12} & 0 \\ \kappa_{12}&\kappa_{22}&0 \\ 0& 0 & \kappa_{33}\end{smallmatrix}\right).
$$
The surface $Q$ is the join of the curve $X =V(\kappa_{11}\kappa_{22}- \kappa_{12}^2 - \kappa_{11}^2 \subset \PP \mathbb{S}^2$ from Example \ref{ex: 2x2ML1 quadric} and the point $Y = V(0) \subseteq \PP \secondvectorspace = \langle \left(\begin{smallmatrix} 0 & 0 & 0 \\ 0&0&0 \\ 0& 0 & 1\end{smallmatrix}\right) \rangle $. The restriction of the determinant to $V$ factors, $\det|_V = (\kappa_{11}\kappa_{22}- \kappa_{12}^2)\kappa_{33}$, where
\begin{align*}
    \mathrm{MLD}_{(\kappa_{11}\kappa_{22}- \kappa_{12}^2)}(X) = 1 \\
    \mathrm{MLD}_{\kappa_{33}}(Y) = 1.
\end{align*}
Now Lemma \ref{lem: dual-planecurve} can be applied to the linear space $V$ and then extended to all of $\mathbb{S}^3$ by projecting to $V$ (as described in \cite[Proposition~2.8]{ML1}) which gives $\mathrm{MLD}_{\det}(Q) = 1$ with
\begin{align*}
    \Phi_{\mathbb{S}^3, Q, \det} =   \Phi_{X, \mathbb{S}^2, \det} \cdot \Phi_{Y, C_Y, \kappa_{33}} = 4\left( \frac{s_{22}}{s_{11}s_{22} - s_{12}^2 + s_{22}^2}\right)^2 \cdot \frac{1}{s_{33}}
\end{align*}
\end{example}

We now turn our attention to a larger class of varieties, those whose dual variety is a nonplanar curve with a special secant line.
\begin{theorem}
\label{thm: degensurface}
Let $X \subset \PP \vectorspace \cong \PP^3$ be the dual variety of a curve $X^\vee$ that is not a plane curve. Suppose that there are two smooth points $[\alpha], [\beta] \in X^\vee$ such that the secant line/pencil they span meets $X^\vee$ in $\mathrm{deg}(X^\vee) -1$ many points. Let $F = \alpha \cdot \beta$ then $\mathrm{MLD}_{F}(X) = 1$, moreover $X$ is ruled by a family of curves $C$ such that $\mathrm{MLD}_{F}(C) = 1$.
\end{theorem}
\begin{proof}
We prove this by proving that both factors of Theorem \ref{thm: F-adjoineddegree} are $1$. By construction, the map $\gamma^\vee_{X,F} $ maps every point of $X$ to a pencil that meets the curve $X^\vee$ and the secant line/pencil spanned by $[\alpha], [\beta]$, denoted by $(\PP\secondvectorspace)^\vee$. The family of lines that meet a curve and a special secant line are known to be of order $1$. To see this, a generic $u$ spans a plane $u + \PP\secondvectorspace^\vee$ which by Bézout's theorem meets $X^\vee$ in a unique point outside $\PP\secondvectorspace^\vee$ that determines a unique line in the family. We now turn our attention to the second factor of the formula in Theorem \ref{thm: F-adjoineddegree}. We must show that $\gamma_{X,F} $ is birational onto its image to conclude the proof. The closure of the fibers of the Gauss map $\gamma_X$ are called contact loci. Bi-duality \cite[Theorem~1.1]{GKZ94} and knowing that $W(X)=W(X^\vee)$ is a projective bundle over the smooth locus of $X^\vee$ implies that the generic contact locus $C$ is a line. To ensure that $\gamma_{X,F} $ is birational we must prove that the restriction of $\nabla F|_C = [\beta(x) \alpha + \alpha(x) \beta] \colon C \to (\PP\secondvectorspace)^\vee$ is birational for a generic $C$. This is equivalent to $F|_C$ having two distinct zeros and by applying the Euler characteristic formula \cite[Theorem~1.3]{DRGS22} this is equivalent to $\mathrm{MLD}_{F}(C) = 1$ . We do a proof by contradiction and assume $F|_C$ has a double root. This in turn is equivalent to the generic contact locus $C$ meeting $\PP\secondvectorspace = V(\alpha) \cap V(\beta)$ which is equivalent to the generic tangent line of $X^\vee$ meeting $(\PP\secondvectorspace)^\vee$. This is a contradiction because then $X^\vee$ would be a plane curve by Terracini lemma \cite[Lemma~1.11]{Ad87} and contradict our assumption. This concludes the proof.
\end{proof}

\begin{example}
\label{ex: surfacedisc}
    Let us consider the projective space $\PP(\mathbb{C}^{4})$ and identify it with its dual vector space via the bilinear form $\langle x, y\rangle = \sum_{i=0}^3 x_iy_i$. Let $X \subset \PP^n$ be the dual variety of the twisted cubic $X^\vee \subset \PP^n$,
    \begin{align*}
        X = &  V(x_1^2x_2^2-4x_0x_2^3-4x_1^3x_3+18x_0x_1x_2x_3-27x_0^2x_3^2) \\
        X^\vee = & V(u_2^2-u_1u_3,u_1u_2-u_0u_3,u_1^2-u_0u_2).
    \end{align*}
The points $(1:0:0:0)$ and $(0:0:0:1)$ lie on $X^\vee$ and we identify them with the linear forms $x_0, x_3 \in X^\vee$. The smooth curve $X^\vee$ is degree $3$ and the secant line spanned by $x_0, x_3$ satisfies the assumptions of Theorem \ref{thm: degensurface}. The corresponding solution to the homaloidal PDE is 
$$
\Phi_{X, \vectorspace, (x_0x_3)}(\bu) = \frac{u_1u_2}{(u_1u_3-u_2^2)(u_0u_2-u_1^2)}
$$
which is verified using the following code in \texttt{Macaulay2} \cite{M2}. 

\begin{figure}[H]
\begin{lstlisting}[language=Macaulay2]
R = QQ[x_0..x_3]; S = QQ[u_0..u_3];
X = x_1^2*x_2^2-4*x_0*x_2^3-4*x_1^3*x_3+18*x_0*x_1*x_2*x_3-27*x_0^2*x_3^2
Phi = (u_1*u_2)/(u_1^2*u_2^2-u_0*u_2^3-u_1^3*u_3+u_0*u_1*u_2*u_3)
p = numerator Phi; gradp = diff(vars S, p)
q = denominator Phi; gradq = diff(vars S, q)
nablaPhi = flatten entries (
    sub((1/q^2), frac S)*sub((q*gradp - p*gradq), frac S))
MLE = apply(nablaPhi, f -> -f/Phi)
MLE_0*MLE_3 == Phi
ker (map(S,R, apply(q^2*nablaPhi, f -> sub(f,S)))) == ideal(X)
\end{lstlisting}
\end{figure}

Notice that this solution is the product of the two functions
$$
\frac{u_2}{(u_0u_2-u_1^2)} \quad \frac{u_1}{(u_1u_3-u_2^2)} 
$$
that are solutions to the homaloidal PDE respectively to $x_0$ and $x_3$ even though the assumptions of Lemma \ref{lem: dual-planecurve} are not satisfied. 
\end{example}

\section{The Generic Maximum Likelihood Degree}
\label{sec: genericMLD}
 
In this section we provide the proof of Theorem \ref{thm: finalpolarformula} and introduce the notion of a variety $X \subset \PP \vectorspace$ being \emph{$F$-general}. We then motivate the name \emph{$F$-general} by proving that a generic perturbation $gX$ of $X$, where $g \in \mathrm{PGL}(\vectorspace)$, is $F$-general. 

\begin{definition}
    \label{def: F-general}
Let $F \colon \vectorspace \dashrightarrow \CC$ be homogeneous of nonzero degree. A variety $X \subset \PP\vectorspace$ is said to be \emph{$F$-general} if $W(X) \cap \overline{\Gamma}_{\nabla F} = \emptyset$ and the variety
\begin{align*}
\{(x,y,z) \colon (x,y) \in W(X) \text{ and } (x,z) \in \overline{\Gamma}_{\nabla F} \} \subset \PP\vectorspace \times \PP \vectorspace^\ast \times \PP \vectorspace^\ast
\end{align*}
is irreducible.
\end{definition}

The intuitive idea behind $X$ being $F$ general is that $X$ should intersect a stratification of the divisor associated with $F$ transversely for the formula to hold. The varieties of maximum likelihood degree $1$ fail at this.

\begin{lemma}
\label{lem: ezfgeneral}
    Suppose $F$ is a non-constant homogeneous polynomial such that $V(F) \subset \PP\vectorspace$ is smooth, then $X$ is $F$-general if and only if 
    $$W(X) \cap \overline{\Gamma}_{\nabla F} = \emptyset.
    $$ 
    Moreover, if $X$ is also smooth this is equivalent to $X$ and $V(F)$ intersecting transversely.  
\end{lemma}
\begin{proof}
    By the assumption that $V(F)$ is smooth, $\nabla F$ is a morphism. Thus $\overline{\Gamma}_{\nabla F} = \Gamma_{\nabla F}$ and the projection
\begin{align*}
        \{(x,y,z) \colon (x,y) \in W(X) \text{ and } (x,z) \in \Gamma_{\nabla F} \} & \to \PP\vectorspace \\
    (x,y,z) &\mapsto x
\end{align*}
    is a projective bundle over $X$ with fibers $\{ x\}\times W_x(X) \times \{ \nabla_xF\}$, hence irreducible. This proves the first part of the claim. For the second part, assuming that $V(F), X$ are smooth and transverse means that 
    $$
    \emptyset  = W(X) \cap W(V(F)) = W(X) \cap \Gamma_{\nabla F}.
    $$
\end{proof}

\begin{example}
    Suppose that $F=Q$ is the Fermat/isotropic quadric $Q= \sum_i x_i^2$ on $\CC^n$. By Lemma \ref{lem: ezfgeneral} a smooth variety $X$ being $Q$-general is equivalent to $X$ intersecting the quadric $V(Q)$ transversely.  
\end{example}

\begin{example}
     Consider the smooth curve $X= V(\kappa_{11}\kappa_{22}-\kappa_{12}^2 - \kappa_{11}^2) \subset \PP \mathbb{S}^2$ from Example \ref{ex: 2x2ML1 quadric}. Letting $F = \det =\kappa_{11}\kappa_{22}-\kappa_{12}^2$ we prove that that $X$ is not $F$-general. Observe that $X$ is tangent to $V(F)$ at $(\begin{smallmatrix} 0 & 0 \\ 0 & 1 \end{smallmatrix} )$ with the common tangent line $V(\kappa_{11})$. The curve $X$ is not $F$-general by Lemma \ref{lem: ezfgeneral}. We may also verify that
     $$
\Big( \left[\begin{smallmatrix} 0 & 0 \\ 0 & 1 \end{smallmatrix} \right], \left[\begin{smallmatrix} 1 & 0 \\ 0 & 0 \end{smallmatrix} \right] \Big) \in W(X) \cap \overline{\Gamma}_{\nabla F} \neq \emptyset. 
     $$
\end{example}

\begin{lemma}
\label{lem: generic-graph-restriction}
    Let $\varphi \colon \PP^n \dashrightarrow \PP^n$ be a rational map and $X\subset \PP^n$. A generic perturbation $gX$ by $g \in \mathrm{PGL}(\vectorspace)$ will respect the graph of $\varphi$ in the sense 
    \begin{align*}
         ((gX) \times \PP^n) \cap \overline{\Gamma}_{\varphi} = \overline{\Gamma}_{(\varphi|_{gX})}  \subset  \PP^n \times \PP^n.
    \end{align*}
\end{lemma}
\begin{proof}
    The inclusion $\overline{\Gamma}_{(\varphi|_{gX})} \subseteq \overline{\Gamma}_{\varphi} \cap (g X \times \PP^n)$ is immediate. Using Kleiman transversality for the product of projective general linear groups on $\PP^n \times \PP^n$ we can say that a general $g$ makes $gX \times \PP^n$ intersect any closed subvariety generically transverse. This means that for any subvariety $Y$ we have 
    $$
    \dim( Y \cap( gX \times \PP^n)) = \dim(Y) -\mathrm{codim}_{\mathbb{P}^n}(X)
    $$ and in particular $\overline{\Gamma}_{\varphi} \cap (gX \times \PP^n)$ is of pure dimension $\mathrm{dim}(X)$. The proof of the lemma will be complete if we prove that the intersection $\overline{\Gamma}_{\varphi} \cap (gX \times \PP^n)$ is irreducible. 
    
    By construction any additional component of $\overline{\Gamma}_{\varphi} \cap (gX \times \PP^n)$ except for $\overline{\Gamma}_{(\varphi|_{gX})}$ must be contained inside $I \times \mathbb{P}^n$ where $I$ is the indeterminacy locus of $\varphi$. It suffices to prove that $\dim \overline{\Gamma}_{\varphi} \cap (((gX) \cap I) \times \PP^n) < \dim(X)$. To see this we observe that
    $$
    \overline{\Gamma_{\varphi}} \supsetneq \overline{\Gamma_{\varphi}} \cap (I \times \PP^n),
    $$
    and since $\overline{\Gamma_{\varphi}} $ is irreducible this means that $\dim(\overline{\Gamma_{\varphi}} \cap (I \times \PP^n)) < n$. This in turn implies that
\begin{align*}
    & \dim \overline{\Gamma_{\varphi}} \cap ((X \cap I) \times \PP^n) = \dim ((\overline{\Gamma_{\varphi}} \cap (I \times \PP^n)) \cap (X \times \PP^n))  \\ =&  \dim \overline{\Gamma_{\varphi}} \cap (I \times \PP^n) - \mathrm{codim}_{\mathbb{P}^n}(X) < n -  \mathrm{codim}_{\mathbb{P}^n}(X) = \dim(X). 
\end{align*}
\end{proof}

\begin{proposition}
\label{prop: pert-F-general}
    Let $X \subset \PP\vectorspace$ and $F\colon \vectorspace \dashrightarrow \CC$ be homogeneous of non-zero degree. For a generic $g \in \mathrm{PGL}(\vectorspace)$, the perturbed variety $gX$ is $F$-general.
\end{proposition}
\begin{proof}
Let $n =\dim \PP \vectorspace$. The first part is to prove that $\overline{\Gamma}_{\nabla F} \cap W(gX) = \emptyset$. Consider the variety 
    $$
V(y(x)) \coloneqq \{([x],[y]) \colon y(x) = 0 \} \subset \PP \vectorspace \times \PP \vectorspace^{\ast}. 
    $$
    This is a smooth $2n-1$ dimensional irreducible variety, as the projection onto either factor yields a projective bundle of hyperplanes. Notice that for any projective variety $X$ we have a containment of the conormal variety $ \Con \subset V(y(x))$. Thus
    $$
\Con \cap \overline{\Gamma}_{\nabla F} \subset  V(y(x)).
    $$
Consider $y(x)$ as a divisor on $\overline{\Gamma}_{\nabla F}$, for any $x$ in the graph and outside the indeterminacy locus of $\nabla F$ we have
$$
\bu(x) =\nabla_xF(x) = \deg(F)F(x)
$$
by Euler's homogeneous function theorem. This means that 
$$
Y = \overline{\Gamma}_{\nabla F} \cap V(y(x)) \subsetneq \overline{\Gamma}_{\nabla F}
$$
so the proof is reduced to proving that 
    $$
    \emptyset = W(g \cdot X) \cap Y \subset V(y(x)).
    $$ 
    Notice that the $\mathrm{PGL}(\vectorspace)$ acts on $V(y(x))$ by 
    $$
    g\cdot(x, y) = (g\cdot x, y \circ g^{-1})
    $$
    and $g\cdot \Con= W(gX)$. This action is transitive on $V(y(x))$ because a linear map can send any hyperplane with a marked point to any another hyperplane with a marked point. Now we observe that 
    \begin{align*}
        \dim V(y(x)) =& 2n-1 \\
        \dim \Con =& n-1 \\
        \dim Y <& \dim \overline{\Gamma}_{\nabla F} = n. 
    \end{align*}
    By Kleiman transversality $W(g X)$ is generically transverse to $Y$ inside of $V(y(x))$. The dimensions of these varieties do not add up so the intersection is empty. 
    
    For the second part of the proof, let 
\begin{align*}
        \Tilde{W}(X) &= \{(x,y,z) \colon (x,y) \in W(X) \} \subset V(y(x)) \times \PP \vectorspace^\ast \\
        \Tilde{\Gamma}_{\nabla F} &= \{(x,y,z) \colon (x,z) \in \overline{\Gamma}_{\nabla F} \} \subset V(y(x)) \times \PP \vectorspace^\ast.
\end{align*}
    Our goal is to show that 
    $$
\Theta_{gX,F} \coloneqq \Tilde{W}(gX) \cap \Tilde{\Gamma}_{\nabla F}
    $$
    is irreducible. Observe that $\mathrm{PGL}(\vectorspace) \times \mathrm{PGL}(\vectorspace)$ acts transitively on $V(y(x)) \times \PP \vectorspace^\ast$ by $$
    (g,h)\cdot(x,y,z) = (g(x), y \circ g^{-1}, z\circ g^{-1})
    $$ and that $(g,h)\cdot \Tilde{W}(X) = \Tilde{W}(gX)$. We conclude that $\Theta_{gX, F}$ is a generically transverse intersection by Kleiman transversality. By Lemma \ref{lem: generic-graph-restriction} we have a map
\begin{align*}
     \pi_{x,z} \colon \Theta_{gX,F} & \to \overline{\Gamma}_{(\nabla F)|_{gX}} \\
     (x,y,z) & \mapsto (x,z)
\end{align*}
    where the codomain is irreducible. Let 
    $$
U \coloneqq \{(x,z) \colon x \in (gX)_\mathrm{reg} \} \subset \overline{\Gamma}_{(\nabla F)|_{gX}}
    $$
     and notice that $\pi_{x,z}$ is a projective bundle over $U$ because $W(gX)$ is a projective bundle over the smooth locus of $gX$. We conclude that any component which is not $\overline{\pi_{x,z}^{-1}(U)}$ must be lying over the singular locus of $gX$. Let
\begin{align*}
    \Tilde{B}(X) &\coloneqq \{(x,y,z) \colon x \in X_{\mathrm{sing}} \} \subsetneq \Tilde{W}(X)
\end{align*}
    and note that $(g,h) \cdot \Tilde{B}(X) = \Tilde{B}(gX)$. Now 
    $$
 \Tilde{B}(gX) \cap \Tilde{\Gamma}_{\nabla F}\subsetneq  \Tilde{W}(gX) \cap \Tilde{\Gamma}_{\nabla F} 
    $$
    are two generically transverse intersections. We conclude that they are of different pure dimension. Any extraneous component of $\Theta_{gX,F} = \Tilde{W}(X) \cap \Tilde{\Gamma}_{\nabla F}$ is contained in the smaller variety $\Tilde{B}(X) \cap \Tilde{\Gamma}_{\nabla F}$. This cannot happen because both varieties have different pure dimensions (equidimensional components). Therefore $\Theta_{gX,F}$ is irreducible. 
\end{proof}

\begin{lemma}
\label{lem: gen-transverse-intersect}
    Let $P$ be an irreducible smooth projective variety and consider two irreducible varieties of the form $A \subset \PP^n \times P, B \subset  P \times \PP^m$ such that 
    $$
    V = (A \times \PP^m) \cap (\PP^n \times B) \subset \PP^n \times P \times \PP^m
    $$ is irreducible and both intersecting varieties are smooth at a generic point of $V$. Let $\pi_A \colon A \to P$ denote the projection and assume it induces a surjective map of tangent spaces at a generic point of $V$. Then the intersection $V$ is generically transverse. 
\end{lemma}
\begin{proof}
    By assumption we have that at a generic point $(x,p,y) \in V$ the involved varieties are smooth, so we have that 
\begin{align*}
        T_{x,p,y}(A \times \PP^m) =T_{x,p}A + T_y\PP^m \\
T_{x,p,y} (\PP^n \times B) = T_x\PP^n + T_{p,y}B
\end{align*}
where $+$ denotes the linear span of all elements in two subsets of the vector space $T_{x,p,y}(\PP^n \times P \times \PP^m) \cong T_{x}(\PP^n) \oplus T_{p}(P) \oplus T_{y}(\PP^m)$. We can now deduce that
    $$
    T_{x,p,y}(A \times \PP^m) + T_{x,p,y} (\PP^n \times B) = T_x\PP^n + (\pi_{A,*}(T_{x,p}A) + \pi_{B,*}(T_{p,y}B ) ) + T_y\PP^m 
    $$
    where $\pi_{A,*}(T_{x,p}A)$ denotes the projection of $T_{x,p}A$ to $T_pP$ via the differential of $\pi_A$, and similarly for $B$. By assumption we have that $\pi_{A,*}(T_{x,p}A) = T_pP$ which completes the proof. 
\end{proof}

\begin{theorem}
\label{thm: MLDpolardegs}
Let $F \colon \vectorspace \dashrightarrow \CC$ homogeneous of nonzero degree and $X \subset \PP \vectorspace \cong \PP^n$ is $F$-general. Then
$$
\mathrm{MLD}_F(X) = \sum_{i=0}^{n-1} \delta_i(X) \mu_{i}(F)
$$
where we use the multidegrees
\begin{align*}
    [\Con] = \sum_{i=0}^{n-1} \delta_i(X)H^{n-i}(H^\vee)^{i+1}  \qquad [\overline{\Gamma}_{\nabla F}] = \sum_{j=0}^n \mu_j(F) H^j (H^\vee)^{n-j} .
\end{align*} 
If $X$ is not $F$-general this is an upper bound. 
\end{theorem}
\begin{proof}
In this proof we identify $\vectorspace$ with $\mathbb{C}^{n+1}$ to simplify notation but keep the distinction between dual vector spaces. Consider the incidence variety 
$$
\gencorr_{X,F} = \{ (x,y,z,u) \colon \, (x,y) \in \Con, \quad (x,z) \in \overline{\Gamma}_{\nabla F}, \quad z \wedge y \wedge u = 0 \} 
$$
as a subvariety of $\PP^n \times (\PP^{n\ast})^3$.
Here $\Con$ is the conormal variety of $X$, $\overline{\Gamma}_{\nabla F}$ is the closed graph of $\nabla F \colon \PP^n \dashrightarrow \PP^{n\ast}$ and recall that $z \wedge y \wedge u =0$ means that $z,y,u$ are collinear. We begin by showing that this variety is irreducible with the given assumptions. Because $X$ is $F$-general the projection of $\gencorr_{X,F}$ away from $u$ is a $\mathbb{P}^1$ bundle over the irreducible variety $\{ (x,y,z) \colon (x,y) \in \Con \text{ and }(x,z) \in \overline{\Gamma}_{\nabla F} \}$ and thus $\gencorr_{X,F}$ is irreducible. If $X$ is not $F$-general then $\gencorr_{X,F}$ may contain extraneous components that can create excess intersections in the computations below, hence the upper bound. 

The projection $(x,y,z,u) \mapsto u$ factors through the projection $\pi_{\vectorspace^*}$ from Definition \ref{def: formal-ML-deg} via the map
\begin{align*}
    \beta \colon & \Tilde{\gencorr}_{X,F} \dashrightarrow \mathfrak{X}_{F}  \\
    & (x,y,z,u) \mapsto (x,u).
\end{align*}
We prove that $\beta$ is birational. The map $\beta$ is dominant because for a generic $(x,u)$ we have that $\{ x \} \times \{ W_x(X) \cap ( \nabla_xF+ u  )  \} \times \{ \nabla_xF\} \times \{ u \} = \beta^{-1}(x,u)$. For a generic $(x,u) \in \mathfrak{X}_{X,F}$ the line $\nabla_xF + u$ meet the projective conormal space $W_x(X)$, in a unique point. This implies $\beta$ is birational. 

It follows with the given assumptions that (with some abuse of notation)
$$
\mathrm{MLD}_F(X) = [\gencorr_{X,F}]\cdot H_u^n.
$$
If the assumptions are not satisfied we instead need to replace $\gencorr_{X,F}$ with one of its irreducible components, and carry out the same argument to illustrate why the formula is an upper bound. Consider the equality 
$$
\gencorr_{X,F} = (\Con \times (\PP^{n*})^2) \cap \pi_{x,z}^{-1}(\overline{\Gamma}_{\nabla F}) \cap (\mathbb{P}^n \times Z),
$$
where $Z$ is the variety of collinear triples $(y,z,u)$ and where we've let $\pi_{x,z}^{-1}(\Gamma_{X,F}) = \{(x,y,z,u) \colon \, (x,z) \in \overline{\Gamma}_{\nabla F} \}$. We proceed by showing that these varieties are generically transverse. First consider the intersection $(\Con \times (\PP^{n*})^2) \cap \pi_{x,z}^{-1}(\overline{\Gamma}_{\nabla F}) $ denoted by $T \times \mathbb{P}^{n*}$ as $u$ can vary freely in both factors. The intersecting varieties are both smooth as long as $x$ is a smooth point of $X$ and $\nabla_x F$ is defined, which is true for a generic point in $T$. Lemma \ref{lem: gen-transverse-intersect} ensures that $T$ is generically transverse because the projection of the graph $\overline{\Gamma}_{\nabla F}$ to the leftmost factor is surjective on the tangent spaces of $\PP^n$.

For a generic point in $\gencorr_{X,F} =T \times \mathbb{P}^{n*} \cap \mathbb{P}^n \times Z$ the two intersecting varieties are smooth because $Z$ is smooth when $y,z,u$ are pairwise distinct. The Lemma \ref{lem: gen-transverse-intersect} ensures $\gencorr_{X,F}$ is generically transverse because the projection of $Z$ away from $u$ is surjective on the tangent spaces of $(\mathbb{P}^{n*})^2$ whenever $y,z,u$ are pairwise distinct. We conclude that 
\begin{align*}
    \mathrm{MLD}_F(X) = [\Con] \cdot [\overline{\Gamma}_{\nabla F}] \cdot [Z]\cdot H_u^n
\end{align*}
where $[\Con], [\overline{\Gamma}_{\nabla F}], [Z]$ denote the pull-backs of these varieties by projection in the Chow ring $\mathbb{Z}[H_x,H_y,H_z,H_u]$. Now we list the classes of these varieties, where the multidegrees of $[Z] \cdot H_u^n$ are $1$ (see \cite[Theorem~5.4]{DHOST16})
\begin{align*}
    [\Con] = \sum_{i=0}^{n-1} \delta_iH_x^{n-i}H_y^{i+1}  \\
    [\overline{\Gamma}_{\nabla F}] = \sum_{j=0}^n \mu_j H_x^j H_z^{n-j} \\
    [Z] \cdot H_u^n = H_u^n \cdot \sum_{k=0}^{n-1}H_y^{n-k-1}H_z^{k}.
\end{align*}
We conclude the proof by computing 
\begin{align*}
    \mathrm{MLD}_F(X) & = [\Con] \cdot [\overline{\Gamma}_{\nabla F}] \cdot [Z]\cdot H_u^n \\
     = & H_u^n \cdot \sum_{k=0}^{n-1}\sum_{j=0}^n\sum_{i=0}^{n-1} \delta_i\mu_j H_x^{n-i}H_y^{i+1}H_x^j H_z^{n-j} H_y^{n-k-1}H_z^{k} \\
     = & H_u^n \cdot \sum_{k=0}^{n-1}\sum_{j=0}^n\sum_{i=0}^{n-1} \delta_i\mu_j H_x^{n-i+j}H_y^{n-k+i} H_z^{n-j+k}
\end{align*}
By recalling that $H^{n+1}=0$ in the Chow ring of $\PP^n$ we observe that we may enforce the conditions
\begin{align*}
    j\leq i \quad i \leq k \quad k \leq j \implies i=k=j
\end{align*}
meaning that 
\begin{align*}
    \mathrm{MLD}_F(X) =  (\sum_{i=0}^{n-1} \delta_i\mu_i) H_u^nH_x^{n}H_y^{n} H_z^{n}
\end{align*}
\end{proof}

\begin{proof}[Proof of Theorem \ref{thm: finalpolarformula}]
    Combine Theorem \ref{thm: MLDpolardegs} and Proposition \ref{prop: pert-F-general}. 
\end{proof}

\begin{example}
    Let $F=Q$ be a non-degenerate quadratic polynomial on $\mathbb{P}^n$. The gradient map of $Q$ is a projective linear map and thus all the multidegrees $\mu_i(Q) =1$. We obtain the formula
    $$
    \mathrm{MLD}_Q(X) = \sum_{i=0}^{n-1} \delta_i(X),
    $$
    for the generic maximum likelihood degree with respect to $Q$. We may apply this formula to the space of all $2 \times 2$ symmetric matrices $\vectorspace = \mathbb{S}^2$. The $2 \times 2$ determinant $(\kappa_{11}\kappa_{22} - \kappa_{12}^2)$ is non-degenerate. If we use the trace pairing $\mathrm{trace}(AB)$ to identify $\mathbb{S}^2 \cong (\mathbb{S}^2)^\ast$ we may write
    $$
    \Gamma_{(\kappa_{11}\kappa_{22} - \kappa_{12}^2)} = \{(\left(\begin{smallmatrix} \kappa_{11} & \kappa_{12}  \\ \kappa_{12}& \kappa_{22}, \end{smallmatrix}\right), \left(\begin{smallmatrix} \kappa_{22} & -\kappa_{12}  \\ -\kappa_{12}& \kappa_{11} \end{smallmatrix}\right)) \} \subset (\PP \mathbb{S}^2)\times (\PP \mathbb{S}^2).
    $$
    Lemma \ref{lem: ezfgeneral} guarantees that a smooth curve $X$ is transversal to the vanishing of the determinant is $\det$-general. The polar degrees of a non-linear plane curve $C$ is $\deg(C)$and $\deg(C^\vee)$. The generic maximum likelihood degree of a nonlinear curve in $\PP \mathbb{S}^2$ can then be computed as
    $$
\mathrm{MLD}_{(\kappa_{11}\kappa_{22} - \kappa_{12}^2)}(C) = \deg(C) + \deg(C^\vee).
    $$
    Let us verify this formula for a general smooth cubic/elliptic curve $C$ by comparing it to the Euler characteristic formula from \cite[Example~4.4]{DRGS22}. The genus of $C$ is $1$ so the Euler characteristic formula yields
    $$
\mathrm{MLD}_{(\kappa_{11}\kappa_{22} - \kappa_{12}^2)}(X) = -\chi(C) + \deg(C) + \#(Q\cap C) = (-2 + 2\cdot 1) + 3 +6 = 9.
    $$
The Plücker formula yields $\deg(C^\vee) = 3(3-1)=6$. Theorem \ref{thm: MLDpolardegs} agrees with the previous Euler characteristic formula
\begin{align*}
    \mathrm{MLD}_{(\kappa_{11}\kappa_{22} - \kappa_{12}^2)}(C) = \deg(C) + \deg(C^\vee) = 3 + 6 = 9. 
\end{align*}
\end{example}

\begin{example}
\label{ex: linearspacesexpected}
    Theorem \ref{thm: MLDpolardegs} suggests that the maximum likelihood degree of a linear space $L \subset \PP\vectorspace$ is exactly a multidegree of the closed graph $\overline{\Gamma}_{\nabla F}$ (and Theorem \ref{thm: finalpolarformula} confirms this). These have been studied in the literature \cite{DMS21} and computed in the case where $F$ is the determinant of a symmetric $m \times m$ matrix \cite{MMW21}.
\end{example}

\begin{example}
Let us apply Theorem \ref{thm: finalpolarformula} to $\mathbb{S}^3$ and $F = \det$. Consider a generic quadric hypersurface $Q \subset \PP \mathbb{S}^3$ and a quintic curve $C \subset \PP \mathbb{S}^3$ cut out by two generic quadrics and two linear forms. The conormal varieties have the multidegrees:
    \begin{align*}
        [W( Q)] =& 2(H_K^5H_S + H_K^4H_S^2 + H_K^3H_S^3 + H_K^2H_S^4 + H_KH_S^5) \\
        [W( C)] =& 8H_K^5H_S + 4H_K^4H_S^2
    \end{align*}
where $H_K,H_S$ generate $A^\ast(\PP \mathbb{S}^3 \times \PP (\mathbb{S}^{3\ast}))$. To see this, the polar degrees of a quadric hypersurface are all $\delta_{i} = 2$ because the Gauss map is just the restriction of a linear map to $Q$. The degree of the dual hypersurface of $C$ is $8$, this can be computed as if $C \subset \PP^3$ was cut out by two quadrics $Q_1,Q_2$. The degree of the dual variety is then the number of points on $C$ satisfying $\det \begin{bmatrix}
        \nabla Q_1 & \nabla Q_2 & \bu & v
    \end{bmatrix} = 0$ which is another quadric, which is expected to intersect $C$ in $8$ points by Bézouts theorem. The multidegrees $\mu_i(\det)$ of the closed graph in $\mathbb{S}^3$ are \cite[Section~5]{AGKMS}
    $$
\left[ \overline{\Gamma}_{\det(K)} \right] =  1H_S^5 +  2H_K^1H_S^4 + 4H_K^2H_S^3 + 4H_K^3H_S^2 +2H_K^4H_S + 1H_K^5 
    $$
     Theorem \ref{thm: MLDpolardegs} predicts the maximum likelihood degree to be
\begin{align*}
    \mathrm{MLD}_{\det}(Q) =& 2(1+2+4+4+2) = 26 \\
\mathrm{MLD}_{\det}(C) =& 1\cdot8 + 2\cdot 4 = 16.
\end{align*}
Knowing that $C$ is an elliptic curve one can also compute $\mathrm{MLD}_{\det}(C) = -2+2+4+12=16$ via the formula \cite[Example~4.4]{DRGS22}. We verify the value of $\mathrm{MLD}_{\det}$ using the following code in \texttt{Macaulay2} \cite{M2}.

{\small
\begin{lstlisting}[language=Macaulay2]
m=3;inds = flatten apply(m,i->apply(i+1,j->(j+1,i+1)));n = #inds -1;
R = QQ[join(apply(inds,I->k_I),apply(inds,I->s_I)), 
    Degrees=>join(apply(n+1,i->{1,0}),apply(n+1,i->{0,1}))]
K = matrix{apply(inds,I->k_I)}
Kmat = matrix apply(m,i->apply(m,j-> k_(min(i,j)+1,max(i,j)+1)))
S = matrix{join(apply(inds,I->s_I), apply(inds,I->0))}
F = ideal det(Kmat)
jac = (J) -> ( transpose jacobian J)
genQ = () -> sum flatten apply(inds,I->apply(inds,J->
    (-1)^(random(ZZ))*random(QQ)*k_I*k_J))
genL = () -> sum apply(inds,I->(-1)^(random(ZZ))*random(QQ)*k_I)
--Choose X
X = ideal(genQ(),genQ(),genL(),genL())
X = ideal(genQ())
c=codim X
--Compute critical points for general data
E = X+minors(c+2,S||jac(F)||jac(X))+ideal(sum apply(inds,I->k_I*s_I)-m);
Kring = QQ[apply(inds,I->k_I)]
genericcrits = saturate(sub(sub(E, 
	    apply(inds,I->s_I=>random(QQ))),Kring), sub(F,Kring));
MLD = degree genericcrits 
\end{lstlisting}
}

\end{example}

\bibliographystyle{alpha}
\bibliography{literature}

\end{document}